\documentclass[a4paper,12pt]{amsart}

\usepackage{amsmath,amstext,amssymb,amsopn,amsthm}
\usepackage{url}
\usepackage{mathtools}

\usepackage[margin=30mm]{geometry}
\usepackage{eucal,mathrsfs,dsfont}

\newtheorem{theorem}{Theorem}[section]
\newtheorem{corollary}[theorem]{Corollary}
\newtheorem{lemma}[theorem]{Lemma}
\newtheorem{proposition}[theorem]{Proposition}

\theoremstyle{definition}

\newtheorem{definition}[theorem]{Definition}

\theoremstyle{remark}
\newtheorem{remark}[theorem]{Remark}

\numberwithin{equation}{section}

\newcommand{\eps}{\varepsilon}

\newcommand{\calA}{\mathcal{A}}
\newcommand{\calF}{\mathcal{F}}
\newcommand{\calG}{\mathcal{G}}
\newcommand{\calD}{\mathcal{D}}
\newcommand{\calGY}{\mathcal{G}^{(Y,S)}}
\newcommand{\calP}{\mathcal{P}}

\newcommand{\I}{\mathds{1}}
\newcommand{\Ee}{\operatorname{\mathds{E}}} 
\newcommand{\Pp}{\operatorname{\mathds{P}}} 
\newcommand{\R}{\mathds{R}}
\newcommand{\real}{\mathds{R}}
\newcommand{\N}{{\mathds{N}}}
\newcommand{\C}{{\mathds{C}}}
\newcommand{\Z}{{\mathds{Z}}}
\newcommand{\Rn}{{\R^n}}
\newcommand{\Rt}{{\R^2}}
\newcommand{\uni}{\mathds{S}}

\newcommand{\term}{\tau}

\newcommand{\wh}{\widehat}
\newcommand{\wt}{\widetilde}

\def\cadlag{c\`adl\`ag\ }

\DeclareMathOperator{\Arg}{Arg}
\DeclareMathOperator{\supp}{supp}

\title[Stationary distributions for jump processes]{Stationary distributions for jump processes with inert drift}
\author{K.\ Burdzy  T.\ Kulczycki \and R.L.\ Schilling}
\address{Krzysztof Burdzy, Department of Mathematics, Box 354350, University of Washington, Seattle, WA 98195, USA}
\email{burdzy@math.washington.edu}
\thanks{K.\ Burdzy was supported in part by NSF Grant DMS-0906743 and by grant N N201 397137, MNiSW, Poland.}
\address{Tadeusz Kulczycki, Institute of Mathematics, Polish Academy of Sciences, ul. Kopernika 18, 51-617 Wroc{\l}aw, Poland \newline Institute of Mathematics and Computer Science, Wroc{\l}aw University of Technology, Wybrzeze Wyspianskiego 27, 50-370 Wroc{\l}aw, Poland}
\email{t.kulczycki@impan.pl}
\thanks{T.\ Kulczycki was supported in part by grant N N201 373136, MNiSW, Poland.}
\address{Rene Schilling, Institut f\"ur Stochastik, TU Dresden, D-01062 Dresden, Germany.}
\email{rene.schilling@tu-dresden.de}
\thanks{R.L.\ Schilling was supported in part by DFG grant Schi 419/5-1.}

\begin{document}

\begin{abstract}
We analyze jump processes $Z$ with ``inert drift'' determined by  a ``memory'' process  $S$. The state space of $(Z,S)$ is the  Cartesian product of the unit circle and the real line.  We prove that the stationary distribution of $(Z,S)$ is the product of the uniform probability measure and a Gaussian distribution.
\end{abstract}

\maketitle

\section{Introduction}\label{sec1}

We are going to find stationary distributions for jump processes with inert drift.
We will first review various sources of inspiration for this project, related models and results. Then we will discuss some technical aspects of the paper that may have independent interest.

 This  paper is concerned with the following system of stochastic differential equations ( the precise statement is in the next section),
\begin{align}\label{5.5.1}
    dY_t  &=  dX_t + W'(Y_t) S_t\, dt, \\
    dS_t  &=   W''(Y_t) \,dt, \label{5.5.2}
\end{align}
where $X$ is a stable  L\'evy  process and $W$ is a $C^5$ function.
This equation is similar to equation \cite[(4.1)]{BBCH}, driven by Brownian motion, but  in \eqref{5.5.1} the term $\frac 12 \, (A \nabla V)(X_t)\, dt$ from the first line of \cite[(4.1)]{BBCH} is missing.  An explanation for  this  can be found in heuristic calculations in  \cite[Example 3.7]{BW}.  The paper \cite{BW}  deals  with Markov processes with finite state spaces and (continuous-space) inert drifts. This class of processes is relatively easy to analyze from the technical point of view. It can be used to generate conjectures, for example,  \cite[Example 3.7]{BW}  contains a conjecture  about  the process defined by \eqref{5.5.1}-\eqref{5.5.2}.

The main result of this paper, i.e.\  Theorem \ref{uniq1}, is concerned with the stationary distribution of a transformation of $(Y,S)$. In order to obtain non-trivial results, we ``wrap'' $Y$ on the unit circle, so that the state space for the transformed process
$Y$
is compact. In other words, we consider
$(Z_t,S_t) = (e^{iY_t},S_t)$. The stationary distribution for
$(Z_t,S_t)$ is the product of the uniform distribution on the circle and the normal distribution.

The product form of the stationary distribution for a two-component Markov process is obvious if the two components are independent Markov processes. The product form is far from obvious if the components are not independent but it does appear in a number of contexts, from queuing theory to mathematical physics. The paper \cite{BW} was an attempt to understand this phenomenon  for a class of models.

One expects to encounter a Gaussian distribution as (a part of) the stationary distribution in some well understood situations. First, Gaussian distributions arise in the context of the Central Limit Theorem (CLT) and continuous limits of CLT-based models. Another class of examples of processes with Gaussian stationary measures comes from mathematical physics. The Gibbs measure is given by $c_1\exp(- c_2 \sum_{i,j} (x_i - x_j)^2)$ in some models, such as the Gaussian free  field, see \cite{shef}.  In such models, the Gaussian nature of the stationary measure arises because the strength of the potential between two elements of the system is proportional to their ``distance'' (as in Hooke's
law for  springs) and, therefore, the potential energy is proportional to the square of the distance between two elements.
Our model is different in that the square in the exponential function represents the ``kinetic energy'' (square of the drift magnitude) and not potential energy of a force. The unexpected appearance of the Gaussian distribution in some stationary measures was noticed in \cite{BW2004} before it was explored more deeply in \cite{BW,BBCH}.

The present article has a companion \cite{BKS} in which
we analyze a related jump
process with ``memory''. In that model, the memory process affects the rate of jumps but it does not add a drift to the jump process. The stationary distribution for that model is also the product of uniform probability measure and a Gaussian distribution.

An ongoing research project of one of the authors is concerned with Markov processes with inert drift when the noise (represented by $X$ in
\eqref{5.5.1}) goes to 0. In other words, one can regard the process $(Y,S)$ as a trajectory of a dynamical system perturbed by a small noise. No matter how small the noise is, the second component of the stationary measure will always be Gaussian. Although we do not study small noise asymptotics in this paper, it is clear from our results that the Gaussian character of the stationary distribution for the perturbed dynamical system does not depend on the Gaussian character of the noise---it holds for the stable noise.

Models of Markov processes with inert drift can represent the motion of an inert particle in a potential, with small noise perturbing the motion. Although such models are related to the Langevin equation (see \cite{Lach}), they are different. There are several recent papers devoted to similar models, see, e.g., \cite{BenLR,BenR1,BenR2,BRO}.

We turn to the technical aspects of the paper.
The biggest effort is directed at determining a core of the generator of the process. This is done by showing that the semigroup $T_t$ of the process $(Y_t,S_t)$ preserves $C_b^2$, see Theorem \ref{Tt}.
The main idea is based on an estimate of the smoothness of the stochastic flow of solutions to \eqref{5.5.1}-\eqref{5.5.2}. This result, proved in greater generality than that needed for our main results, is presented in Section \ref{sec3}, see Proposition \ref{path}. This proposition actually makes an assertion on the \emph{pathwise} smoothness of the flow.
It seems that Theorem \ref{Tt} and Proposition \ref{path} are of independent interest.

\subsection{Notation}
Since the paper uses a large amount of notation, we collect most frequently used symbols in the table below, for easy reference.

\bigskip\noindent
\begin{center}
\begin{footnotesize}
\begin{tabular}{l|p{10cm}}
$a\vee b$, $a\wedge b$
            & $\max(a,b)$, $\min(a,b)$;\\
\hline \\
$a_+$, $a_-$
            & $\max(a,0)$, $-\min(a,0)$;\\ \hline \\
$|x|_{\ell^1}$
            & $\displaystyle \sum_{j=1}^m |x_j|$ where $x=(x_1,\ldots, x_m)\in\R^m$;\\ \hline \\
$e_k$ & the $k$-th unit base vector in the usual orthonormal basis for $\R^n$;\\ \hline \\
$\calA_{\alpha}$
            & $\displaystyle \alpha \Gamma\left(\frac{1 + \alpha}2\right) \frac{2^{\alpha - 1}}{\sqrt\pi\,\Gamma\big(1-\frac\alpha2\big)}$, \; $\alpha\in (0,2)$;\\ \hline \\
$D^\alpha$
            & $\displaystyle \frac{\partial^{|\alpha|}}{\partial x_1^{\alpha_1} \cdots \partial x_d^{\alpha_d}}$, \; $\alpha=(\alpha_1,\ldots,\alpha_d)\in\N_0^d$;\\ \hline \\
$C^k$       & $k$-times continuously differentiable functions;\\ \hline \\
$C^k_b$, $C^k_c$, $C^k_0$
            & functions in $C^k$ which, together with all their derivatives up to order $k$, are ``bounded'', are ``compactly supported'', and ``vanish at infinity'',  respectively; \\ \hline \\
$\|f\|_{\infty,B}$
            & $\displaystyle \sup_{x\in B}|f(x)|$ for  $f:\R^n\to\R$;\\ \hline \\
$\|D^{(j)} f\|_{\infty,B}$
            & $\displaystyle \sum_{|\alpha|=j} \|D^\alpha f\|_{\infty,B}$;\\ \hline \\
$\|f\|_{(j),B}$, $\|f\|_{(j)}$
            & $\displaystyle \sum_{|\alpha|\leq j} \sup_{x\in B} |D^\alpha f(x)|$, resp.\ $\displaystyle\sum_{|\alpha|\leq j} \|D^\alpha f\|_\infty$;\\ \hline \\
$\|D^{(j)} V\|_{\infty,B}$, $\|D^{(j)} V\|_{\infty}$
            & $\displaystyle \sum_{|\alpha|=j}\sum_{k=1}^n \sup_{x\in B} |D^\alpha V_k(x)|$, resp., $\displaystyle \sum_{|\alpha|=j}\sum_{k=1}^n   ||D^\alpha V_k||_{\infty}$ for any function $V:\R^n\to\R^n$;\\ \hline \\
$\|V\|_{(j),B}$, $\|V\|_{(j)}$
            & $\displaystyle \sum^j_{i= 0 } \|D^{(i)} V\|_{\infty, B}$, resp., $\displaystyle \sum^j_{i= 0 } \|D^{(i)} V\|_{\infty}$;\\ \hline \\
$\uni$      & $\{z\in \C \::\: |z|=1\}$ unit circle in $\C$.
\end{tabular}
\end{footnotesize}
\end{center}

\medskip\noindent
Constants $c$ without sub- or superscript are generic and may change their value from line to line.

\section{A jump process with a smooth drift}\label{sec2}

Let $\uni = \{z \in \C \::\: |z| = 1\}$ be the unit circle in $\C$. Consider a $C^5$ function $V: \uni \to \R$
which is not identically constant and put $W(x) = V(e^{ix})$, $x \in \R$. Let $X_t$ be a symmetric $\alpha$-stable  L\'evy process on $\R$ which has the jump density $\calA_{\alpha}\,|x - y|^{-1-\alpha}$, $\alpha \in (0,2)$. Let $(Y,S)$ be a Markov process with the state space $\Rt$ satisfying the following SDE,
\begin{equation}\label{sdemain}
    \begin{cases}
    \displaystyle
    dY_t  =  dX_t + W'(Y_t) S_t\, dt, \\[\medskipamount]
    \displaystyle
    dS_t  =   W''(Y_t) \,dt.
    \end{cases}
\end{equation}

\begin{lemma}\label{existence}
    The SDE \eqref{sdemain} has a unique strong solution which is a strong Markov process with c\`adl\`ag paths.
\end{lemma}
\begin{proof}
For every $n \in \N$ define the function $f_n: \R \to \R$ by $f_n(s) := (-n)\vee s\wedge n$. We consider for fixed $n \in \N$ the following SDE
\begin{equation}\label{sden}
    \begin{cases}
        \displaystyle
        dY_t^{(n)}  =  dX_t + W'(Y_t^{(n)}) f_n(S_t^{(n)})\, dt, \\[\medskipamount]
        \displaystyle
        dS_t^{(n)}  =  W''(Y_t^{(n)})\, dt.
    \end{cases}
\end{equation}
Note that $\Rt \ni (y,s) \mapsto W'(y) f_n(s)$ is a Lipschitz function. By \cite[Theorem V.7]{Pr} and \cite[Theorems V.31, V.32]{Pr} the SDE \eqref{sden} has a unique strong solution which has the strong Markov property and c\`adl\`ag paths for every fixed $n \in \N$.

Now fix $t_0 < \infty$ and a starting point $\Rt \ni (y,s) = (Y_0^{(n)},S_0^{(n)} )$. Note that for any $t \le t_0$ we have
$$
    \left|S_t^{(n)}\right|
    = \left|S_0^{(n)} + \int_0^t W''(Y_s^{(n)}) \, ds\right|
    \leq |s| + t_0 \|W''\|_{\infty}.
$$
Pick $n > |s| + t_0 \|W''\|_{\infty}$, $n \in \N$. For such $n$ and any $t \le t_0$, the process $(Y_t,S_t) := (Y_t^{(n)},S_t^{(n)} )$ is a solution to \eqref{sdemain} with starting point $(y,s)$.
 This shows that for any fixed starting point $(y,s) = (Y_0,S_0)$ and fixed $t_0 < \infty$ the SDE \eqref{sdemain} has a unique strong solution up to time 
$t_0$. The solution
 is strong Markov and has c\`adl\`ag paths. Since $t_0 < \infty$ and the starting point $(y,s)$  are
 arbitrary, the lemma follows.
\end{proof}

We will now introduce some notation. Let
$\N$ be  the  positive integers and  $\N_0 = \N \cup \{0\}$.
 For any $f: \uni \to \R$ we set
$$
    \tilde{f}(x) := f(e^{ix}), \quad x \in \R.
$$
We say that $f: \uni \to \R$ is \emph{differentiable} at $z= e^{ix}$, $x \in \R$, if and only if $\tilde{f}$ is differentiable at $x$ and we put
$$
    f'(z) := (\tilde{f})'(x), \quad \text{where} \quad  z = e^{ix}, \quad x \in \R.
$$
Analogously, we say that $f: \uni \to \R$ is \emph{$n$ times differentiable}  at $z= e^{ix}$, $x \in \R$, if and only if $\tilde{f}$ is $n$ times differentiable at $x$ and we write
$$
    f^{(n)}(z) = (\tilde{f})^{(n)}(x), \quad \text{where} \quad  z = e^{ix}, \quad x \in \R.
$$
In a similar way we define for $f: \uni \times \R \to \R$
\begin{equation}\label{ftilde2}
    \tilde{f}(y,s) = f(e^{iy},s), \quad y,s \in \R.
\end{equation}
We say that $D^\alpha f(z,s)$, $z = e^{iy}$, $y,s \in \R$, $\alpha\in\N_0^2$, exists if and only if $D^\alpha\tilde{f}(y,s)$ exists and we set
$$
    D^\alpha f(z,s) = D^\alpha\tilde{f}(y,s), \quad \text{where} \quad  z = e^{iy}, \quad  y,s \in \R.
$$
When writing
 $C^2(\uni)$, $C_c^2(\uni \times \R)$,
etc., we are referring to the derivatives defined above.

Let
\begin{align}\label{Zdef}
Z_t = e^{i Y_t}.
\end{align}
Then $(Z,S)$ is ``a symmetric $\alpha$-stable process with inert drift wrapped on the unit circle''.
In general, a function of a (strong) Markov process is not any longer a Markov process. We will show that the ``wrapped'' process $(Z_t,S_t)=(e^{iY_t},S_t)$ is a strong Markov process because the function $W(x) = V(e^{ix})$ is periodic.

\begin{lemma}\label{sol-periodic}
    Let $(Y_t,S_t)$ be the solution of the SDE \eqref{sdemain}. Then
    $$
        \Pp^{(y + 2 \pi,s)}(Y_t \in A + 2 \pi, \, S_t \in B) = \Pp^{(y,s)}(Y_t \in A, \, S_t \in B)
    $$
    holds for all $(y,s)\in\R^2$ and all Borel sets $A,B\subset\R$.
\end{lemma}
\begin{proof}
    Denote by $(Y_t^y,S_t^s)$ the unique solution of the SDE \eqref{sdemain} with initial value $(Y_0^y,S_0^s)=(y,s)$. We assume without loss of generality that $X_0 \equiv 0$.  By definition, the process $(Y_t^{y+2\pi},S_t^s)$ solves
    \begin{equation*}
    \begin{cases}
    \displaystyle
    \hat{Y}_t =  y + 2\pi + X_t + \int_0^t W'(\hat{Y}_r) \hat{S}_r \, dr,\\[\medskipamount]
    \displaystyle
    \hat{S}_t = s + \int_0^t W''(\hat{Y}_r) \, dr.
    \end{cases}
\end{equation*}
Since the function $W$ is periodic with period
$2 \pi$, we know that
 $W'(\hat{Y}_r) = W'(\hat{Y}_r - 2 \pi)$ and $W''(\hat{Y}_r) = W''(\hat{Y}_r - 2 \pi)$. Therefore, $(Y_t^{y+2\pi},S_t^s)$ solves the system
\begin{equation*}
    \begin{cases}
    \displaystyle
    \hat{Y}_t =  y + 2\pi + X_t + \int_0^t W'(\hat{Y}_r - 2 \pi) \hat{S}_r \, dr,\\[\medskipamount]
    \displaystyle
    \hat{S}_t = s + \int_0^t W''(\hat{Y}_r - 2 \pi) \, dr.
    \end{cases}
\end{equation*}
By subtracting $2 \pi$ from both sides of the first equation we get
\begin{equation*}
    \begin{cases}
    \displaystyle
    \hat{Y}_t - 2 \pi =  y +  X_t + \int_0^t W'(\hat{Y}_r - 2 \pi) \hat{S}_r \, dr,\\[\medskipamount]
    \displaystyle
    \hat{S}_t = s + \int_0^t W''(\hat{Y}_r - 2 \pi) \, dr.
    \end{cases}
\end{equation*}
Since the solutions are unique, this shows that 
$(Y_t^{y+2\pi},S_t) = (Y_t^{y}+2\pi,S_t)$
 from which the claim follows.
\end{proof}
We can now use a rather general result on transformations of the state space due to Dynkin \cite[10.25, Theorem 10.13]{dyn1}, see also Glover \cite{glo} and Sharpe \cite[Section 13]{sha}.

\begin{corollary}\label{wrap-mp}
Let $\gamma : \R^2 \to \uni\times\R$, $\gamma(y,s):= (e^{iy},s)$ and $(Y_t,S_t)$ be the unique, c\`adl\`ag strong Markov solution of
the SDE \eqref{sdemain}. Then $(Z_t,S_t)=(e^{iY_t},S_t)$ is also a
strong Markov process. Let $P_t((y,s),A\times B)$
denote the transition function of $(Y,S)$ and
 $P^{\uni}_t((y,s),A\times B)$
the transition function  of $(Z,S)$.
Then for $y,s \in \R$ and Borel sets $A,B\subset \R$,
    \begin{gather*}
        P_t^\uni(\gamma(y,s),A\times B) = P_t((y,s),\gamma^{-1}(A\times B))
    \end{gather*}
\end{corollary}

\begin{proof}
    All we have to do is to verify Dynkin's condition \cite[10.25.A]{dyn1} saying that
   $$
        P_t((y,s),\gamma^{-1}(A\times B)) = P_{ t }((y',s'),\gamma^{-1}(A\times B))
    $$
    holds for all Borel sets  $A \subset \uni$, $B\subset\R$ and all points $(y,s),(y',s')\in\R^2$ such that $\gamma(y,s)=\gamma(y',s')$. Clearly, $s=s'$ and $y-y' = 2j\pi$ for some $j\in \Z$.  Denote $f(y) = e^{iy}$ . Applying Lemma \ref{sol-periodic} repeatedly we find
    \begin{align*}
        \Pp^{(y,s)}\big((Y_t,S_t)\in \gamma^{-1}(A\times B)\big)
        &=  \Pp^{(y,s)}\big(Y_t\in  f^{-1}(A),\, S_t\in B\big)\\
        &=  \Pp^{(y + 2\pi j,s)}\big(Y_t\in  f^{-1}(A)  + 2 \pi j,\, S_t\in B\big)\\
        &=  \Pp^{(y + 2\pi j,s)}\big(Y_t\in  f^{-1}(A) ,\, S_t\in B\big)\\
        &= \Pp^{( y  +2\pi j, s)}\big((Y_t,S_t)\in \gamma^{-1}(A\times B)\big).
    \qedhere
    \end{align*}
\end{proof}

We are going to calculate  the  generators
of the processes $X_t$,  $(Y_t,S_t)$ and $(Z_t,S_t)$.

By $\calG^X$ let us denote the generator of the semigroup, 
defined
 on the Banach space $(C_b(\R),\|\cdot \|_\infty)$,
 of the process $X_t$. By $\calD(\calG^X)$ we denote the domain of $\calG^X$. It is well known that $C_b^2(\R) \subset \calD(\calG^X)$ and for $f \in C_b^2(\R)$  we have $\calG^X f = -(-\Delta)^{\alpha/2} f$, where
$$
-(-\Delta)^{\alpha/2} f(x) = \calA_{\alpha} \lim_{\eps \to 0^+} \int_{|y - x| > \eps} \frac{f(y) - f(x)}{|x - y|^{1 + \alpha}} \, dy, \quad x \in \R.
$$
If $f \in C_b^2(\R)$ is periodic with period $2 \pi$ then we have
\begin{equation}\label{periodic}\begin{aligned}
    -(-\Delta)^{\alpha/2} f(x)
    &=  \calA_{\alpha} \lim_{\eps \to 0^+} \int_{\pi > |y - x| > \varepsilon} \frac{f(y) - f(x)}{|x - y|^{1 + \alpha}} \, dy\\
    &\qquad+
    \calA_{\alpha} \sum_{n \in \Z \setminus \{0\}} \int_{\pi > |y - x|} \frac{f(y) - f(x)}{|x - y + 2n\pi|^{1 + \alpha}} \, dy.
\end{aligned}\end{equation}

In the sequel we will need the following auxiliary notation
\begin{definition}\label{class}
\begin{align*}
    C_*(\Rt)
    &:= \big\{f:\Rt \to \R \:: \: \exists N > 0 \: 
\supp(f) \subset \R \times [-N,N], \\
    &\qquad\qquad f \text{\ \ is bounded and uniformly
continuous on\ \ } \Rt \big\},\\
    C_*^2(\Rt) &:= C_*(\Rt) \cap C_b^2(\Rt).
\end{align*}
\end{definition}

 Let us define the
 transition semigroup $\{T_t\}_{t \ge 0}$ of the process $(Y_t,S_t)$  by
\begin{equation}\label{semigroupZS}
    T_tf(y,s) = \Ee^{(y,s)}f(Y_t,S_t), \quad y,s \in \R,
\end{equation}
for functions $f\in C_b(\R^2)$.
Let $\calGY$ be the generator of $\{T_t\}_{t \ge 0}$
and let
 $\calD(\calGY)$
be the
 domain of $\calGY$.
\begin{lemma}\label{gen1}
    We have $C_*^2(\Rt) \subset \calD(\calGY)$ and for $f \in C_*^2(\Rt)$,
\begin{equation}\label{GZS}
    \calGY f(y,s) =  -(-\Delta_y)^{\alpha/2} f(y,s) + W'(y) s f_y(y,s) + W''(y) f_s(y,s), \quad y,s \in \R.
\end{equation}
\end{lemma}
\begin{proof}
Let $f \in C_*(\Rt)$. Throughout the proof we will assume that $\supp(f) \subset \R \times (-M_0,M_0)$ for some $M_0 > 0$. Note that for any starting point
$(Y_0,S_0)=(y,s) \in \R \times [-M_0,M_0]$
 and all $0 \le t \le 1$,
$$
    |S_t| = \left|S_0 + \int_{0}^t W''(Y_r) \, dr \right| \le M_0 + \|W''\|_{\infty}.
$$
Put
$$
M_1 = M_0 + \|W''\|_{\infty}.
$$
Note that if $(y,s) \notin \R \times [-M_1,M_1]$ and $(Y_0,S_0) = (y,s)$ then for any $0 \le t \le 1$ we have
$$
    |S_t| = \left|S_0 + \int_{0}^t W''(Y_r) \, dr \right| > M_1 -  \|W''\|_{\infty}  = M_0,
$$
and, therefore,
 $f(Y_t,S_t) = 0$. It follows that for any $(y,s) \notin \R \times [-M_1,M_1]$ and $0 < h \le 1$ we have
$$
    \frac{\Ee^{(y,s)}f(Y_h,S_h) - f(y,s)}{h} = 0.
$$

We may, therefore, assume that $(y,s) \in \R \times [-M_1,M_1]$. We will also assume that
 $0 < h \le 1$.

As above we see that for any starting point
$(Y_0,S_0) =(y,s) \in \R \times [-M_1,M_1]$ and all $0 \le t \le 1$ we have
 $|S_t| \le M_1 + \|W''\|_{\infty}$. Set $M_2 := M_1 + \|W''\|_{\infty}$.
We assume without loss of generality that $X_0 \equiv 0$.
Then
\begin{align*}
    Y_t &= y + X_t + \int_0^t W'(Y_r) S_r \, dr, \\
    S_t &= s + \int_0^t W''(Y_r) \, dr.
\end{align*}
It follows that
\begin{align*}
    \frac{T_h f(y,s) - f(y,s)}{h}
    &= \frac{\Ee^{(y,s)} f(Y_h,S_h) - f(y,s)}{h} \\
    &= \frac{1}{h} \Ee^{(y,s)}[f(Y_h,S_h) - f(Y_h,s)] + \frac{1}{h}\Ee^{(y,s)}[f(Y_h,s) - f(y,s)] \\
    &= \text{I} + \text{II}.
\end{align*}
Using Taylor's theorem we find
\begin{align*}
    \text{I}
    &= \Ee^{(y,s)}\left[ \frac{1}{h} \frac{\partial f}{\partial s} (Y_h,s) \int_0^h W''(Y_r) \, dr + \frac{1}{2 h} \frac{\partial^2 f}{\partial s^2} (Y_h,\xi) \left(\int_0^h W''(Y_r) \, dr \right)^2 \right] \\
    &= \Ee^{(y,s)}\Bigg[ \frac{1}{h} \frac{\partial f}{\partial s} (Y_h,s) \int_0^h W''(y) \, dr + \frac{1}{h} \frac{\partial f}{\partial s} (Y_h,s) \int_0^h (W''(Y_r) - W''(y))\, dr \\
    &\qquad\qquad\quad+ \frac{1}{2 h} \frac{\partial^2 f}{\partial s^2} (Y_h,\xi) \left(\int_0^h W''(Y_r) \, dr \right)^2 \Bigg],
\end{align*}
where $\xi$ is a point between $s$ and $S_h$. Note that
\begin{align*}
    \Ee^{(y,s)}
    &\left[ \left| \frac{1}{h} \frac{\partial f}{\partial s} (Y_h,s) \int_0^h \big(W''(Y_r) - W''(y)\big)\, dr \right| \right] \\
    &\le \Ee^{(y,s)}\left[ \frac{1}{h} \left\|\frac{\partial f}{\partial s}\right\|_{\infty} \int_0^h  \left\{ \left( \|W'''\|_\infty \left|X_r + \int_0^r W'(Y_t) S_t \, dt \right| \right) \wedge 2 \|W''\|_{\infty} \right\}  \, dr \right] \\
    &\le \left\|\frac{\partial f}{\partial s}\right\|_{\infty}  \Ee^{(y,s)}  \left[\left\{ \|W'''\|_{\infty} \left( \sup_{0 \le r \le h} |X_r| + h \|W'\|_{\infty} M_2 \right) \right\}\wedge 2 \|W''\|_{\infty} \right]\\
    &\xrightarrow[h\to 0^+]{} 0,
\end{align*}
uniformly for all  $(y,s) \in \R \times [-M_1,M_1]$. The convergence follows from the right continuity of $X_t$ and our assumption that $X_0 = 0$. We also have
$$
    \Ee^{(y,s)}  \left[ \left|\frac{1}{2 h} \frac{\partial^2 f}{\partial s^2} (Y_h,\xi) \left(\int_0^h W''(Y_r) \, dr \right)^2\right| \right]
    \le \left\|\frac{\partial^2 f}{\partial s^2}\right\|_{\infty} \frac{h}{2}\,
\|W''\|_{\infty}^2
    \xrightarrow[h\to 0^+]{} 0,
$$
uniformly for all  $(y,s) \in \R \times [-M_1,M_1]$. Because $Y_h$ is right-continuous it is easy to see that
$$
    \Ee^{(y,s)}\left[ \frac{1}{h} \frac{\partial f}{\partial s} (Y_h,s) \int_0^h W''(y) \, dr  \right]
    \xrightarrow[h\to 0^+]{}
    \frac{\partial f}{\partial s}(y,s) W''(y),
$$
uniformly for all  $(y,s) \in \R \times [-M_1,M_1]$. It follows that
$$
    \text{I} \xrightarrow[h\to 0^+]{} \frac{\partial f}{\partial s}(y,s) W''(y),
$$
uniformly for all  $(y,s) \in \R \times [-M_1,M_1]$.

Now let us consider $\text{II}$. We have
\begin{align*}
    \text{II}
    &= \frac{1}{h}\Ee^{(y,s)}[f(y + X_h,s) - f(y,s)] + \frac{1}{h}\Ee^{(y,s)}[f(Y_h,s) - f(y + X_h,s)] \\
    &= \text{II}_1 + \text{II}_2.
\end{align*}
It is well known that
$$
    \text{II}_1 \xrightarrow[h\to 0^+]{} -(-\Delta_y)^{\alpha/2} f(y,s),
$$
uniformly for all  $(y,s)$. We also have
\begin{align*}
    \text{II}_2
    &= \Ee^{(y,s)}\left[ \frac{1}{h} \frac{\partial f}{\partial y}(y + X_h,s) \int_0^h W'(Y_r) S_r \, dr+ \frac{1}{2 h} \frac{\partial^2 f}{\partial y^2}(\xi,s) \left(\int_0^h W'(Y_r) S_r \, dr \right)^2 \right] \\
    &= \Ee^{(y,s)}\Bigg[ \frac{1}{h} \frac{\partial f}{\partial y}(y + X_h,s) \left(\int_0^h W'(y) s \, dr + \int_0^h W'(Y_r) (S_r - s) \, dr \right.\\
    &\qquad\qquad\quad \mbox{} + \left. \int_0^h (W'(Y_r)-W'(y)) s \, dr\right)
    +\frac{1}{2 h} \frac{\partial^2 f}{\partial y^2}(\xi,s) \left(\int_0^h W'(Y_r) S_r \, dr     \right)^2 \Bigg],
\end{align*}
where $\xi$ is a point between $y + X_h$ and $Y_h$. Using similar arguments as above we obtain
$$
    \text{II}_2 \xrightarrow[h\to 0^+]{} \frac{\partial f}{\partial y}(y,s) W'(y) s,
$$
uniformly for all  $(y,s) \in \R \times [-M_1,M_1]$.

It follows that
$$
    \frac{T_h f(y,s) - f(y,s)}{h} 
    \xrightarrow[h\to 0^+]{} -(-\Delta_y)^{\alpha/2} f(y,s) + W'(y) s \,\frac{\partial f}{\partial y}(y,s) + W''(y)\, \frac{\partial f}{\partial s}(y,s),
$$
uniformly for all  $(y,s) \in \R \times [-M_1,M_1]$. This means that $f \in \calD(\calGY)$ and \eqref{GZS} holds.
\end{proof}

\begin{remark}
A weaker version of Lemma \ref{gen1} can be proved as follows. If we rewrite the SDE \eqref{sdemain} in the form
$$
    d\begin{pmatrix} Y_t \\ S_t\end{pmatrix}
    = \begin{pmatrix} 1 & W'(Y_t) S_t \\ 0 & W''(Y_t) \end{pmatrix} d\begin{pmatrix} X_t \\ t\end{pmatrix}
    = \Phi(Y_t,S_t)\, d\begin{pmatrix} X_t \\ t\end{pmatrix}
$$
and notice that $(X_t,t)^\top$ is a two-dimensional L\'evy process with characteristic exponent $\psi(\xi,\tau)=|\xi|^\alpha + i\tau$, we can use \cite[Theorem 3.5, Remark 3.6]{sch-sch} to deduce that $C_c^\infty(\real^2)\subset \calD(\calG^{(Y,S)})$. This argument
uses the fact
 that the SDE has only jumps in the direction of the $\alpha$-stable process, while it is local in the other direction. Theorem 3.1 of \cite{sch-sch} now applies and shows that $\calG^{(Y,S)}$ is a pseudo-differential operator $\calG^{Y,S} u (x,s) = (2\pi)^{-2}\int_{\real^2} p(x,s;\xi,\tau)\, \calF u(\xi,\tau)\, e^{ix\xi + is\tau}\,  d\xi \, d\tau $,
where $\calF$
 denotes the Fourier transform, with symbol
$$
    p(x,s; \xi,\tau)
    = \psi(\Phi(y,s)^\top (\xi,\tau)^\top)
    = |\xi|^\alpha + i\xi W'(x)s.
$$
A Fourier inversion argument now shows that \eqref{GZS} holds for $f \in C_c^\infty(\real^2)$ and by a standard closure argument we deduce from this that \eqref{GZS} also holds for $f\in C_0^2(\real^2)$.
\end{remark}

\bigskip
We say that $f \in C_0(\uni \times \R)$ if and only if for every $\varepsilon >0 $ there exists a compact set $K \subset \uni \times \R$ such that $|f(u)| < \varepsilon$ for $u \in K^c$.  Let us define the semigroup $\{T_t^\uni\}_{t \ge 0}$ of the process $(Z_t,S_t)$ by
\begin{equation}\label{definitionTtD}
    T_t^\uni f(z,s) = \Ee^{(z,s)}f(Z_t,S_t), \quad z \in \uni, \quad s \in \R,
\end{equation}
for $f$ belonging to $C_0(\uni \times \R)$.  Let $z = e^{iy}$, $y \in \R$. For future reference, we note the following consequences of Corollary \ref{wrap-mp},
\begin{equation}\label{Tttilde}
    T_t^\uni f(z,s)
    = \Ee^{(z,s)}f(Z_t,S_t)
    = \Ee^{(y,s)}f(e^{i Y_t},S_t)
    = \Ee^{(y,s)}\tilde{f}(Y_t,S_t)
    = T_t \tilde{f}(y,s),
\end{equation}
and
\begin{equation}\label{doubletilde}
    \widetilde{T_t^\uni f}(y,s) = T_t \tilde{f}(y,s).
\end{equation}

By $\Arg(z)$ we denote the argument of $z \in \C$ contained in $(-\pi,\pi]$. For $g \in C^2(\uni)$ let us put
\begin{equation}\label{Lz1}\begin{aligned}
    Lg(z)
    &=  \calA_{\alpha} \lim_{\eps \to 0^+} \int_{\uni \cap \{|\Arg(w/z)| > \eps\}} \frac{g(w) - g(z)}{|\Arg(w/z)|^{1 + \alpha}} \, dw \\
    &\qquad+\calA_{\alpha} \sum_{n \in \Z \setminus \{0\}} \int_{\uni} \frac{g(w) - g(z)}{|\Arg(w/z) + 2 n \pi|^{1 + \alpha}} \, dw,
\end{aligned}\end{equation}
where $\calA_\alpha$ is the constant appearing in \eqref{periodic} and $dw$ denotes the arc length measure on $\uni$; note that $\int_\uni \,dw = 2 \pi$.

 Let $\calG$ be
 the generator of the semigroup $\{T_t^\uni\}_{t \ge 0}$
and let $\calD(\calG)$ be
 its domain.

\begin{lemma}\label{gen2}
    We have $C_c^2(\uni \times \R) \subset \calD(\calG)$ and for $f \in C_c^2(\uni \times \R)$,
    \begin{equation*}
        \calG f(z,s)
        = L_z f(z,s) + V'(z) s f_z(z,s) + V''(z) f_s(z,s), \quad z \in \uni, \quad s \in \R.
    \end{equation*}
\end{lemma}
\begin{proof}
Let $f \in C_c^2(\uni \times \R)$. Note that $\tilde{f} \in C_*^2(\Rt)$. We obtain from \eqref{GZS}, for  $z = e^{iy}$, $y,s \in \R$,
\begin{align}\notag
&\lim_{t \to 0^+} \frac{T_t^\uni f(z,s) - f(z,s)}{t}
= \lim_{t \to 0^+} \frac{T_t \tilde{f}(y,s) - \tilde{f}(y,s)}{t} \\
\label{generator1}
&\phantom{\lim_{t \to 0^+}}\qquad =  -(-\Delta)^{\alpha/2} \tilde{f}(y,s) + W'(y) s \tilde{f}_y(y,s) + W''(y) \tilde{f}_s(y,s).
\end{align}
By Lemma \ref{gen1} this limit exists uniformly in $z$ and $s$, i.e.\  $f \in \calD(\calG)$.

We get from \eqref{periodic}
\begin{equation}
\label{Lz3}
-(-\Delta_y)^{\alpha/2}\tilde{f}(y,s) = L_z f(z,s).
\end{equation}
Recall that we have
$W(y) = V(e^{iy}) $,
 $y \in \R$. Using our
definitions
 we get $V'(z) = W'(y)$, $V''(z) = W''(y)$ for $z = e^{iy}$, $y \in \R$. Hence \eqref{generator1} equals
$$
    L_z f(z,s) + V'(z) s f_z(z,s) + V''(z) f_s(z,s),
$$
which gives the assertion of the lemma.
\end{proof}

We will need the following auxiliary lemma.

\begin{lemma}
\label{L}
For any  $f \in C^2(\uni)$ we have
$$
\int_\uni Lf(z) \, dz = 0.
$$
\end{lemma}

\begin{proof}
Recall that $\Arg(z)$ denotes the argument of $z \in \C$ belonging to $(-\pi,\pi]$.
First we will
 show that
\begin{equation}\label{intL}
    \iint_{\uni\times \uni} \I_{\{w\::\:|\Arg(w/z)| > \eps\}}(w) \, \frac{f(w) - f(z)}{|\Arg(w/z)|^{1 + \alpha}} \, dw \, dz = 0.
\end{equation}
We
 interchange $z$ and $w$, use
Fubini's theorem
 and observe that $|\Arg(z/w)| = |\Arg(w/z)|$,
\begin{align*}
    \iint_{\uni\times \uni} &\I_{\{w\::\:|\Arg(w/z)| > \eps\}}(w) \, \frac{f(w) - f(z)}{|\Arg(w/z)|^{1 + \alpha}} \, dw \, dz\\
    &= \iint_{\uni\times \uni}  \I_{\{z\::\:|\Arg(z/w)| > \eps\}}(z)  \, \frac{f(z) - f(w)}{|\Arg(z/w)|^{1 + \alpha}} \, dz \, dw\\
    &= \iint_{\uni\times \uni}  \I_{\{z\::\:|\Arg(z/w)| > \eps\}}(z)  \, \frac{f(z) - f(w)}{|\Arg(z/w)|^{1 + \alpha}} \, dw \, dz\\
    &= -\iint_{\uni\times \uni} \I_{\{w\::\:|\Arg(w/z)| > \eps\}}(w) \, \frac{f(w) - f(z)}{|\Arg(w/z)|^{1 + \alpha}} \, dw \, dz,
\end{align*}
which proves \eqref{intL}.

By interchanging $z$ and $w$ we also get that
\begin{equation}\label{intL2}\begin{aligned}
\sum_{n \in \Z \setminus \{0\}} &\int_\uni \int_{\uni} \frac{f(w) - f(z)}{|\Arg(w/z) + 2 n \pi|^{1 + \alpha}} \, dw \, dz \\
& = \sum_{n \in \Z \setminus \{0\}} \int_\uni \int_{\uni} \frac{f(z) - f(w)}{|\Arg(z/w) + 2 n \pi|^{1 + \alpha}} \, dz \, dw.
\end{aligned}\end{equation}
Note that for $\Arg(w/z) \ne \pi$ we have $|\Arg(z/w) + 2 n \pi| = |\Arg(w/z) - 2 n \pi|$. Hence the expression in \eqref{intL2} equals $0$.

Set
$$
    L_{\eps}f(z) := \int_{\uni \cap \{|\Arg(w/z)| > \eps\}} \frac{f(w) - f(z)}{|\Arg(w/z)|^{1 + \alpha}} \, dw.
$$
What is left is to show that
\begin{equation}\label{intLeps}
    \int_\uni \lim_{\eps \to 0^+} L_{\eps}f(z) \, dz = \lim_{\eps \to 0^+} \int_\uni L_{\eps}f(z) \, dz.
\end{equation}
By the Taylor expansion we have for $f \in C^2(\uni)$
$$
    f(w) - f(z) = \Arg(w/z) f'(z) + \Arg^2(w/z) r(w,z), \quad w, \, z \in \uni,
$$
where  $|r(w,z)| \le c(f)$.  Hence,
\begin{align*}
    |L_{\eps}f(z)|
    &= \left|\int_{\uni \cap \{|\Arg(w/z)| > \eps\}} r(w,z) \Arg^{1 - \alpha}(w/z) \, dw\right| \\
    &\le c(f) \int_\uni |\Arg^{1 - \alpha}(w/z)| \, dw = c(f,\alpha).
\end{align*}
Therefore, we get \eqref{intLeps} by the bounded convergence theorem.
\end{proof}

We will
 identify the stationary measure for $(Z_t,S_t)$.
\begin{proposition}\label{statmeasure}
For $z \in \uni$ and $s \in \R$ let
$$
    \rho_1(z) \equiv \frac{1}{2 \pi},
    \quad \rho_2(s) = \frac{1}{\sqrt{2 \pi}} \, e^{-s^2/2},
    \quad \pi(dz,ds) = \rho_1(z) \rho_2(s) \, dz \, ds.
$$
Then for any  $f \in C_c^2(\uni \times \R)$ we have
$$
\int_\uni \int_\R \calG f(z,s) \, \pi(dz,ds) = 0.
$$
\end{proposition}

\begin{proof} We have
\begin{align*}
    \int_\uni \int_\R &\calG f(z,s) \, \pi(dz,ds)\\
    &= \frac{1}{2 \pi} \int_\uni \int_\R \big(L_z f(z,s) + V'(z) s f_z(z,s) + V''(z) f_s(z,s) \big) \rho_2(s) \, ds \, dz.
\end{align*}
Integrating by parts, we see that this is equal to
\begin{gather*}
    \frac{1}{2 \pi} \int_\uni \int_R L_z f(z,s) \rho_2(s) \, ds \, dz
    - \frac{1}{2 \pi} \int_\uni \int_R V''(z) s f(z,s) \rho_2(s) \, ds \, dz
    \\
    - \frac{1}{2 \pi} \int_\uni \int_R V''(z) f(z,s) {\rho_2}'(s) \, ds \, dz
    = \text{I} + \text{II} + \text{III}.
\end{gather*}
Since ${\rho_2}'(s) = -s \rho_2(s)$ we
find that
 $\text{II} + \text{III} = 0$, while $\text{I} = 0$ by Lemma \ref{L}. The claim follows.
\end{proof}

\begin{proposition}\label{TtD}
For any $t \ge 0$ we have
$$
    T_t^\uni: C_c^2(\uni \times \R) \to C_c^2(\uni \times \R).
$$
\end{proposition}
The proof of this proposition is quite difficult. It is deferred to the next section in which we prove this result in much greater generality for solutions of SDEs driven by L\'evy processes.

\begin{theorem}\label{th:m3.1}
Let
\begin{align}\label{5.5.5}
    \pi(dz,ds) = \frac{1}{(2 \pi)^{3/2}}\, e^{-s^2/2} \, dz \, ds, \quad z \in \uni, \quad s \in \R.
\end{align}
Then $\pi$ is a
stationary distribution of the process $(Z_t,S_t)$.
\end{theorem}
\begin{proof}
Let $(Y_t,S_t)$ be a Markov process satisfying the SDE \eqref{sdemain} and $(Z_t,S_t)=(e^{i Y_t},S_t)$. Recall that $\{T_t^\uni\}_{t \ge 0}$ is the semigroup on $C_0(\uni \times \R)$ defined by \eqref{definitionTtD} and $\calG$ is its generator.  Let $\calP(\R \times \R)$ and $\calP(\uni \times \R)$ denote the sets of all probability measures on $\R \times \R$ and $\uni  \times \R$ respectively. In this proof, for any $\tilde\mu \in \calP(\uni \times \R)$ we define $\mu \in \calP(\R \times \R)$ by $\mu([0,2\pi)\times \R) = 1$ and $\mu(A \times B) = \tilde\mu(e^{iA} \times B)$ for Borel sets $A \subset [0,2\pi)$, $B \subset \R$.

Consider any $\tilde\mu \in \calP(\uni \times \R)$ and the corresponding $\mu \in \calP(\R \times \R)$.

For this $\mu$ there exists a Markov process $(Y_t,S_t)$ given by \eqref{sdemain} such that $(Y_0,S_0)$ has the distribution $\mu$.
It follows that for any $\wt\mu \in \calP(\uni \times \R)$ there exists a Markov process $(Z_t,S_t)$ given by \eqref{sdemain} and $Z_t = e^{i Y_t}$ such that $(Z_0,S_0)$ has the distribution $\wt\mu$. By Proposition 4.1.7 \cite{EK}, $(Z_t,S_t)$ is a solution of the martingale problem for $(\calG,\wt\mu)$. The Hille-Yosida theorem shows that the assumptions  of Theorem 4.4.1 \cite{EK} are satisfied if we take $A = A' = \calG$. Thus Theorem 4.4.1 \cite{EK} implies that for any $\wt\mu \in \calP(\uni \times \R)$, uniqueness holds for the martingale problem for $(\calG,\wt\mu)$. Hence the martingale problem for $\calG$ is well posed.

Note that $C_c^2(\uni \times \R)$ is dense in $C_0(\uni \times \R)$, that is, in the set on which the semigroup $\{T_t^\uni\}_{t \ge 0}$ is defined. It follows from Proposition \ref{TtD} and Proposition 1.3.3 from 
\cite{EK} that $C_c^2(\uni \times \R)$ is a core for $\calG$. Now using Proposition \ref{statmeasure} and Proposition 4.9.2 from \cite{EK} we get that $\pi$ is a stationary measure for $\calG$. This means that $(Z_t,S_t)$ has a stationary distribution $\pi$.
\end{proof}

\begin{theorem}\label{uniq1}
The measure $\pi$ defined in \eqref{5.5.5} is the unique
stationary distribution of the process $(Z_t,S_t)$.
\end{theorem}
\begin{proof}

Suppose that for some \cadlag processes $X^1$ and $X^2$, processes  $(Y_t^1,S_t^1)$ and $(Y_t^2,S_t^2)$ satisfy
\begin{align}
Y^1_t &=  y + X^1_t + \int_0^t W'(Y^1_r) S^1_r \, dr, \label{5.6.10}\\
S^1_t &= s + \int_0^t W''(Y^1_r) \, dr, \label{5.6.11}\\
Y^2_t &=  y + X^2_t + \int_0^t W'(Y^2_r) S^2_r \, dr, \label{5.6.12}\\
S^2_t &= s + \int_0^t W''(Y^2_r) \, dr. \label{5.6.13}
\end{align}
Then
\begin{align}\label{5.6.5}
|S^1_t - S^2_t| \leq \int_0^t |W''(Y^1_r)-W''(Y^2_r)| \, dr
\leq \|W^{(3)}\|_\infty \int_0^t |Y^1_r-Y^2_r| \,dr,
\end{align}
and, therefore, for $t\leq 1$,
\begin{align*}
&|Y^1_t - Y^2_t| \leq |X^1_t - X^2_t| + \int_0^t |W'(Y^1_r)S^1_r-W'(Y^2_r)S^2_r| \, dr\\
&\leq |X^1_t - X^2_t| + \int_0^t |W'(Y^1_r)(S^1_r - S^2_r)| \, dr
+ \int_0^t |(W'(Y^1_r)-W'(Y^2_r))S^2_r| \, dr\\
&\leq |X^1_t - X^2_t|
+ \|W'\|_\infty \sup_{0\leq r \leq t} |S^1_r-S^2_r| \, t
+ \|W''\|_\infty
\sup_{0\leq r \leq t} |S^2_r| \int_0^t |Y^1_r-Y^2_r| \,dr\\
&\leq |X^1_t - X^2_t|
+ \|W'\|_\infty \,t\, \|W^{(3)}\|_\infty \int_0^t |Y^1_r-Y^2_r| \,dr   \\
&\quad \quad + \|W''\|_\infty \left( |s| + \|W''\|_\infty t \right)
\int_0^t |Y^1_r-Y^2_r| \,dr\\
&\leq |X^1_t - X^2_t|
+ (c_1 t + c_2 |s|) \int_0^t |Y^1_r-Y^2_r| \,dr \\
&\leq |X^1_t - X^2_t|
+ (c_1  + c_2 |s|) \int_0^t |Y^1_r-Y^2_r| \,dr .
\end{align*}
By Gronwall's inequality,
\begin{align*}
\sup_{0\leq r \leq t} |Y^1_r-Y^2_r|
&\leq \sup_{0\leq r \leq t}|X^1_r - X^2_r|
+ \int_0^t |X^1_r - X^2_r| (c_1  + c_2 |s|)
\exp\big\{(c_1  + c_2 |s|)\,t\big\}\, dr\\
&\leq \sup_{0\leq r \leq t}|X^1_r - X^2_r|
\big(1 +   t (c_1  + c_2 |s|)
\exp\big\{(c_1  + c_2 |s|)t\big\} \big).
\end{align*}
For $t =1$, the inequality becomes
\begin{align}\label{5.6.14}
\sup_{0\leq r \leq 1} |Y^1_r-Y^2_r|
\leq \sup_{0\leq r \leq 1}|X^1_r - X^2_r|
\big(1 +   (c_1  + c_2 |s|)
\exp\big\{(c_1  + c_2 |s|)\big\} \big).
\end{align}

We substitute \eqref{5.6.11} into \eqref{5.6.10} and rearrange terms to obtain,
\begin{align*}
X^1_t &=  -y + Y^1_t -
\int_0^t W'(Y^1_r)
\left(s + \int_0^r W''(Y^1_u) \, du\right) \, dr.
\end{align*}

We substitute the (non-random) number $y$ for $Y^1_t$ in the above formula to obtain
\begin{align}\label{6.11.1}
X^1_t &=  -y + y -
\int_0^t W'(y)
\left(s + \int_0^r W''(y) \, du\right) \, dr\\
 &= -W'(y) (t s + t^2 W''(y)/2).
\nonumber
\end{align}
From now on, $X^1$ will denote the process defined
in \eqref{6.11.1}. It is easy to see that $X^1_t$ is well defined for all $t\geq 0$. If we substitute this $X^1$ into \eqref{5.6.10}--\eqref{5.6.11}  then $Y_t \equiv y$.

It follows from  \cite[Theorem II, p.~9]{Sim}, that every continuous function is in the support of the distribution of the symmetric $\alpha$-stable L\'evy process on $\R$.
We will briefly outline how to derive the last claim from the much more general result in \cite[Theorem II, p.~9]{Sim}. One should take $a(\,\cdot\,) \equiv 0$ and $b(\,\cdot\, , z) \equiv z$. Note that the ``skeleton'' functions in \cite[(5), p.~9]{Sim} can have jumps at any times and of any sizes so the closure of the collection of all such functions in the Skorokhod topology contains the set of all continuous functions. Standard arguments then show that every continuous function is in the support of the distribution of the stable process also in the topology of uniform convergence on compact time intervals.
We see that if $X^1$ is the continuous function defined in \eqref{6.11.1} and
$X^2_t$ is a stable process as in \eqref{sdemain} then for every $\eps>0$ there exists $\delta>0$ such that,

\begin{align*}
\Pp\left(\sup_{0\leq r \leq 1}|X^1_r - X^2_r| \leq \eps \right)
\geq \delta.
\end{align*}
This and \eqref{5.6.14} show that for any $y,s\in \R$ and
$\eps>0$ there exists $\delta>0$ such that,
\begin{align*}
\Pp^{y,s}\left(\sup_{0\leq r \leq 1}|X^1_r - X^2_r| \leq \eps,
\sup_{0\leq r \leq 1}|Y^2_r - y| \leq \eps \right)
\geq \delta.
\end{align*}
Note that $S$ can change by at most $\|W''\|_\infty$ on any
interval of length 1. This, the Markov property and induction
show that for any $\eps>0$ there exist $\delta_k>0$, $k\geq 1$,
such that,
\begin{align*}
\Pp^{y,s}\left(\sup_{k\leq r \leq k+1}|X^1_r - X^2_r| \leq 2^{-k}\eps ,
\sup_{k\leq r \leq k+1}|Y^2_r - Y^2_{k}| \leq 2^{-k}\eps \right)
\geq \delta_k.
\end{align*}
where
$X^1$ is defined in \eqref{6.11.1}.
This
implies that for any $\tau < \infty$, $y,s\in \R$ and $\eps>0$
there exists $\delta'>0$ such that,
\begin{align}\label{5.6.20}
\Pp^{y,s}\left(\sup_{0\leq r \leq \tau}|X^1_r - X^2_r| \leq 2\eps,
\sup_{0\leq r \leq \tau}|Y^2_r - y| \leq 2\eps \right)
\geq \delta'.
\end{align}

\noindent
\emph{Step 2}. Recall that $V$ is not identically
constant. This and the fact  that $V \in C^5$ easily imply that
$W''$ is strictly positive on some interval and it is strictly
negative on some other interval. We fix some $a_1, a_2\in(-\pi,\pi)$,
$b_1>0$, $b_2<0$ and $\eps_0\in(0,\pi/100)$, such that
$V''( z )>b_1$ for $z\in \uni$, $\Arg(z) \in [a_1-4\eps_0, a_1 +
4\eps_0]$, and $V''( z)<b_2$ for $z\in \uni$, $\Arg(z) \in
[a_2-4\eps_0, a_2 + 4\eps_0]$.

Suppose that there exist two stationary probability
distributions $\pi$ and $\wh\pi$ for $(Z,S)$. Let $((Z_t,
S_t))_{t\geq0}$ and $((\wh Z_t, \wh S_t))_{t\geq0}$ be
processes with $(Z_0,S_0)$ and $(\wh Z_0, \wh S_0)$ distributed according to
$\pi$ and $\wh\pi$,  respectively. The transition probabilities for
these processes are the same as for the processes defined by
\eqref{sdemain} and \eqref{Zdef}. Let $X$ denote the driving stable L\'evy process for $Z$.

Let $A$ be an open set such that $W''(y) >c >0$ for all $y\in A$. In view of the relationship between $V$ and $W$, we can assume that $A$ is periodic, that is, $y\in A$ if and only if $y+2\pi \in A$. It follows easily from \eqref{sdemain} that there exist $q_1>0$ and $s_1< \infty$ such that for any $(Y_0,S_0)$, the process $Y$ enters $A$ at some random time $T_1\leq s_1 $ with probability greater than $q_1$. Since $Y$ is right continuous, if $Y_{T_1} \in A$ then $Y_t$ stays in $A$ for all $t$ in some interval $(T_1, T_2)$, with $T_2 \leq 2 s_1$. Then \eqref{sdemain} implies that $S_t \ne 0$ for some $t\in (T_1, T_2)$. A repeated application of the Markov property at  the  \emph{}times $2s_1, 4s_1, 6s_1, \dots$ shows that the probability that $S_t = 0$ for all $t \leq 2k s_1$ is less than $(1-q_1)^k$. Letting $k\to \infty$, we see that
$S_t\ne0$ for some $t>0$, a.s.

Suppose without loss of generality that there exist $\eps_1>0$, $t_2>0$ and $p_1>0$ such that $\Pp^\pi(S_{t_2}>\eps_1) > p_1$. Let $F_1 = \{S_{t_2}>\eps_1\}$ and $t_3 = \eps_1/(2\|W''\|_\infty)$. It is easy to see that for some $p_2>0$,
$$
    \Pp^{\pi}  \left(\exists\, t\in [t_2, t_2+t_3]\::\: \Arg(Z_t)\in [a_2-\eps_0, a_2 + \eps_0] \:\big| \: F_1 \right) > p_2.
$$

This implies that there exist $\eps_1>0$, $t_2>0$, $t_4 \in [t_2, t_2+t_3]$ and $p_3>0$ such that,
\begin{align*}
\Pp^\pi(S_{t_2}>\eps_1, \Arg(Z_{t_4}) \in [a_2-2\eps_0, a_2 + 2\eps_0]) > p_3.
\end{align*}
Note that $|S_{t_4} - S_{t_2}| \leq \|W''\|_\infty t_3 < \eps_1/2$. Hence,
\begin{align*}
\Pp^\pi(S_{t_4}>\eps_1/2, \Arg(Z_{t_4}) \in [a_2-2\eps_0, a_2 + 2\eps_0]) > p_3.
\end{align*}
Let $\eps_2 \in( \eps_1/2,\infty)$ be such that
\begin{align*}
\Pp^\pi(S_{t_4}\in[\eps_1/2,\eps_2], \Arg(Z_{t_4}) \in [a_2-2\eps_0, a_2 + 2\eps_0]) > p_3/2.
\end{align*}
Let $t_5 = 2\eps_2 /|b_2|$ and $t_6 = t_4 + t_5$. By \eqref{5.6.20}, for any $\eps_3>0$ and some $p_4>0$,
\begin{align*}
\Pp^\pi\Big(&\sup_{t_4\leq r \leq t_6}|X^1_r - X_r| \leq \eps_3,
S_{t_4}\in[\eps_1/2,\eps_2],\\
& \Arg(Z_{t}) \in [a_2-3\eps_0, a_2 + 3\eps_0] \text{\ \ for  all\ \  } t\in [t_4,  t_6]\Big) > p_4,
\end{align*}
where $X^1$ is
 the function
 defined in \eqref{6.11.1}.
Since $V''( z) < b_2 < 0$ for $ \Arg z \in [a_2-3\eps_0, a_2 + 3\eps_0]$, if the
event in the last formula holds then
\begin{align*}
S_{t_6} = S_{t_4} + \int _{t_4}^{t_6} V''(Z_s) \, ds
\leq \eps_2 + b_2 t_5 \leq -\eps_2.
\end{align*}
This implies that,
\begin{align}\label{5.6.21}
\Pp^\pi\Big(\sup_{t_4\leq r \leq t_6}|X^1_r - X_r| \leq \eps_3,
S_{t_4} \geq \eps_1/2, S_{t_6} \leq - \eps_2\Big) > p_4.
\end{align}

\noindent \emph{Step 3}.  By the L\'evy-It\^o representation we can write the stable L\'evy process $X$ in the form $X_t = J_t + \wt X_t$, where $J$ is a compound Poisson process comprising all jumps of $X$ which are greater than $\eps_0$ and $\wt X = X-J$ is an independent L\'evy process (accounting for all small jumps of $X$).  Let us denote by $\lambda = \lambda(\alpha,\eps_0)$ the rate of the compound Poisson process $J$.

 Let $(\wt Y, \wt S)$ be the solution to \eqref{sdemain}, with $X_t$ replaced by $\wt X_t$ for $t\geq t_4$. Take $\eps_3 < \eps_0/2$.
Then
$\sup_{t_4\leq r \leq t_6}|X^1_r - \wt X_r| \le \eps_3$ entails that $\sup_{t_4\leq r \leq t_6}|J_{t_4} -J_r|=0$. Thus, \eqref{5.6.21} becomes
\begin{align*}
    &\Pp^\pi\Big(\sup_{t_4\leq r \leq t_6}| X_r^1 - \wt X_r| \leq \eps_3,\:  \wt S_{t_4} \geq \tfrac{\eps_1}2,\: \wt S_{t_6} \leq - \eps_2\: \Big) \\
    &\geq \Pp^\pi\Big(\sup_{t_4\leq r \leq t_6}| X_r^1 - \wt X_{r}| \leq \eps_3,\: \sup_{t_4\leq r \leq t_6}| J_{t_4} - J_{r}|=0,\:  \wt S_{t_4} \geq \tfrac{\eps_1}2,\: \wt S_{t_6} \leq - \eps_2\: \Big) \\
    &> p_4>0.
\end{align*}

Let $\tau$ be the time of the first jump of $J$ in the
interval $[t_4, t_6]$; we set $\tau=t_6$ if there is no such
jump.  We can represent $\{(Y_t,S_t), 0\leq t \leq \tau\}$ in the following way, $(Y_t,S_t) = (\wt Y_t, \wt S_t)$ for  $0\leq t < \tau$, $S_\tau = \wt S_\tau$, and $Y_\tau = \wt Y_\tau + J_{\tau} - J_{\tau-}$.

We say that a non-negative measure $\mu_1$ is a component of a non-negative measure $\mu_2$ if $\mu_2 = \mu_1 + \mu_3$ for some non-negative measure $\mu_3$. Let $\mu(dz,ds) = \Pp^\pi(Z_\tau\in dz, S_\tau \in ds)$. We will argue that $\mu(dz, ds)$ has a component with a density bounded below by $c_2 >0$ on $\uni \times (-\eps_2, \eps_1/2)$.  We find  for every Borel set $A\subset\uni$ of arc length $|A|$ and every interval $(s_1,s_2) \subset (-\eps_2, \eps_1/2)$
\begin{align*}\small
    &\mu(A\times (s_1,s_2))\\
    &= \Pp^\pi\left(Z_\tau\in A,\: S_\tau\in (s_1,s_2)\right)\\
    &\geq \Pp^\pi\Big(Z_\tau\in A, S_\tau\in (s_1,s_2),  \sup_{t_4\leq r \leq t_6}| X_r^1 - \wt X_{r}| \leq \eps_3,\:  \wt S_{t_4} \geq \tfrac{\eps_1}2,\:  \wt S_{t_6} \leq - \eps_2\:  \Big)\\
    &\geq \Pp^\pi\Big(e^{i(J_\tau- J_{\tau-})}\in e^{-i\wt X_{\tau-}} A,\: \wt S_\tau\in (s_1,s_2),\\
    &\qquad \sup_{t_4\leq r \leq t_6}| X_r^1 - \wt X_{r}| \leq \eps_3,\wt S_{t_4} \geq \eps_1/2,\: \wt S_{t_6} \leq - \eps_2,\: N^J=1\Big).
\end{align*}
Here $N^J$ counts the number of jumps of the process $J$ occurring during the interval $[t_4,t_6]$. Without loss of generality we can assume that $\eps_0 < 2\pi$. In this case the density of the jump measure of $J$ is bounded below by $c_3>0$ on $(2\pi,4\pi)$. Observe that the processes $(\wt X, \wt S)$ and $J$ are independent. Conditional on $\{N^J=1\}$, $\tau$ is uniformly distributed on $[t_4,t_6]$, and the probability of the event $\{N^J = 1\}$ is $\lambda (t_6 - t_4) e^{-\lambda (t_6 - t_4)}$. Thus,
\begin{align*}
    &\mu(A\times (s_1,s_2))  \\
    \geq & c_3 |A| \Pp^\pi\Big(\wt S_\tau\in (s_1,s_2) \Big|\: \sup_{t_4\leq r \leq t_6}| X_r^1 - \wt X_{r}| \leq \eps_3, \wt S_{t_4} \geq \eps_1/2, \wt S_{t_6} \leq - \eps_2,  N^J = 1\Big)\\
    &\qquad  \times   p_4\cdot \lambda (t_6 - t_4) e^{-\lambda (t_6 - t_4)}.
\end{align*}Since the process $\wt S$ spends at least $(s_2-s_1)/\|W''\|_\infty$  units of time in $(s_1,s_2)$ we finally arrive at
$$
    \mu(A,  (s_1,s_2) ) \geq  p_4 \lambda e^{-\lambda (t_6 - t_4)}  c_3 |A|  (s_2-s_1)/\|W''\|_\infty.
$$
This proves that $\mu(dz, ds)$ has a component with a density bounded below by $c_2=  p_4 \lambda e^{-\lambda (t_6 - t_4)}  c_3/\|W''\|_\infty $ on $\uni \times (-\eps_2, \eps_1/2)$.

\medskip
\noindent \emph{Step 4}. Let $ \eps_4 = \eps_1/2 \land
\eps_2>0$. We have shown that for some stopping time $\tau$, $
\Pp^\pi(Z_\tau\in dz, S_\tau \in ds)$ has a component with a
density bounded below by $c_2>0$ on $\uni \times (-\eps_4,
\eps_4)$. We can prove in an analogous way that for some
stopping time $\wh\tau$ and $\wh \eps_4>0$, $ \Pp^{\wh\pi}(\wh
Z_{\wh\tau}\in dz, \wh S_{\wh\tau} \in ds)$ has a component
with a density bounded below by $\wh c_2>0$ on $\uni \times
(-\wh\eps_4, \wh\eps_4)$.

Since $\pi \ne \wh\pi$, there exists a Borel set $A\subset \uni
\times \R$ such that $\pi(A) \ne \wh \pi(A)$. Moreover, since
any two stationary probability measures are either mutually
singular or identical,  cf.\ \cite[Chapter 2, Theorem 4]{S}, we have $\pi(A)>0$ and $\wh\pi(A) =0$
for some $A$. By the strong Markov property applied at $\tau$
and the ergodic theorem, see \cite[Chapter 1, page 12]{S}, we have $\Pp^\pi$-a.s.
\begin{align*}
    \lim_{t\to \infty} (1/t) \int_\tau^t \I_{\{(Z_s, S_s) \in A\}}\,ds = \pi(A)>0.
\end{align*}
Similarly, we see that $\Pp^{\wh\pi}$-a.s.
\begin{align*}
\lim_{t\to \infty} (1/t) \int_{\wh\tau}^t \I_{\{(\wh Z_s,\wh S_s) \in A\}}\,ds
= \wh\pi(A)=0.
\end{align*}
Since the distributions of $( Z_{\tau},  S_{\tau})$ and $(\wh
Z_{\wh\tau}, \wh S_{\wh\tau})$ have mutually absolutely
continuous components,  the last two statements contradict each
other. This shows that we must have $\pi = \wh\pi$.
\end{proof}

\begin{remark}
 It is not hard to show that Theorem \ref{th:m3.1} holds even if we take $\alpha = 2$ in \eqref{sdemain}, that is, if $X_t$ is Brownian motion. It seems that for $\alpha = 2$ uniqueness of the stationary distribution can be proved using techniques employed in Proposition 4.8 in \cite{BBCH}. A close inspection of the proofs in this section reveals that our results remain also valid if $X_t$ is a symmetric L\'evy process with jump measure having full support.
\end{remark}

\section{Smoothness of $T_t f$}\label{sec3}

In this section, we will show that if $f \in C_b^2$ then $T_t f \in C_b^2$ where $\{T_t\}_{t \ge 0}$ is the semigroup of a process defined by a stochastic differential equation driven by a L\'evy process. We use this result to show Proposition \ref{TtD} but it may well be of independent interest. We found some related results in the literature but none of them was sufficiently strong for our purposes. The key element of the proof are explicit bounds for derivatives of the flow of  solutions to the SDE. This is done in Proposition \ref{path}. We provide a direct and elementary proof of this proposition. Note that our bounds are non-random and do not depend on the sample path. This is a new feature in this type of analysis since usually, see e.g.\ Kunita \cite{kunita}, the constants are random since they are derived with the Kolmogorov-Chentsov-Totoki lemma or a Borel-Cantelli argument. Let us, however, point out that there is an alternative way of proving Proposition \ref{path}. It is possible to use \cite[Theorems V.39, V.40]{Pr} and \cite[formula (D), p.\ 305]{Pr} to obtain bounds for derivatives of the flow. Since this alternative approach demands similar arguments and is not shorter than our proof of Proposition \ref{path}, we decided to prove Proposition \ref{path} directly.

Consider the following system of stochastic differential equations in $\Rn$,
\begin{equation}\label{sdeY}
    \begin{cases}
        \displaystyle
        dY_1(t)  =  dX_1(t) + V_1(Y(t))\, dt, \\[\medskipamount]
        \displaystyle
        \quad\vdots \\[\medskipamount]
        \displaystyle
        dY_n(t)  =  dX_n(t) + V_n(Y(t))\, dt,
    \end{cases}
\end{equation}
where $Y(t) = (Y_1(t),\ldots,Y_n(t)) \in \Rn$, $X(t) = (X_1(t),\ldots,X_n(t)) \in \Rn$. We assume that $X(0) = 0$, $X_1, \ldots, X_n$ are L\'evy processes on $\R$ and $V_i: \Rn \to \R$ are locally Lipschitz. We allow $X_1,\ldots,X_n$ to be degenerate, i.e.\ some or all $X_i$ may be identically equal to $0$.

By \cite[Theorem V.38]{Pr} it follows that if $Y(0) = x$ then there exists a stopping time $\zeta(x,\omega): \Rn \times \Omega \to [0,\infty]$ and there exists a unique solution of \eqref{sdeY} with $Y(0) = x$ with $\limsup_{t \to \zeta(x,\cdot)} |Y(t)| = \infty$ a.s.\ on $\zeta < \infty$; $\zeta$ is called the \emph{explosion time}. In order to apply \cite[Theorem V.38]{Pr} we take in the equations marked $(\otimes)$ in \cite[p.\ 302]{Pr} $m = n + 1$, $X_t^i = Y_i(t)$, $x^i = Y_i(0)$, $Z_t^{\alpha} = X_{\alpha}(t)$ for $\alpha \in \{1,\ldots,n\}$, $Z_t^{n + 1} = t$ and $f_{\alpha}^i = \delta_{\alpha i}$ for $\alpha, i \in \{1,\ldots,n\}$ and $f_{n + 1}^i(x) = V_i(x)$ for $i \in \{1,\ldots,n\}$.

By $Y^x(t)$ we denote the process with starting point $Y^x(0) = x$.
In the rest of this section, we will assume that \eqref{sdeY} holds not only a.s.\ but for all $\omega \in \Omega$. More precisely, we can and will assume that the solution to \eqref{sdeY} is constructed on a probability space $\Omega$ such that $X(0)=0$ and
\begin{align*}
Y^x(t) = x + X(t) + \int_0^t V(Y(s))\,ds,
\end{align*}
for all $t\geq 0$ and all $\omega \in \Omega$.

Set
$$
    \|x\| = \max\{|x_1|,\ldots,|x_n|\}, \quad x = (x_1,\ldots,x_n),
$$
and
$$
    B^*(x,r) = \{y \in \Rn \::\:  \|y - x\| < r\}, \quad x \in \Rn,\; r > 0.
$$
For $f: \Rn \to \R$ and $A \subset \Rn$ we write $D^{(1)}f = \nabla f$,
\begin{gather*}
    \|f\|_{\infty,A} = \sup_{x \in A} |f(x)|,
\qquad
    \|D^{(j)} f\|_{\infty,A}
    =  \sum_{|\alpha|=j} \sup_{x \in A} |D^\alpha f(x)|,\\
    \|f\|_{(j),A} = \|f\|_{\infty,A} + \|D^{(1)}f\|_{\infty,A} + \ldots + \|D^{(j)}f\|_{\infty,A}.
\end{gather*}
When $A = \Rn$ we drop $A$ from this notation. For $V = (V_1,\ldots,V_n)$ from \eqref{sdeY} and $A \subset \Rn$ we put
\begin{gather*}
    \|V\|_{\infty,A} = \sum_{i=1}^n \|V_i\|_{\infty,A},\quad
    \|D^{(j)}V\|_{\infty,A} = \sum_{i=1}^n \|D^{(j)} V_i\|_{\infty,A}. \\
    \|V\|_{(j),A} = \|V\|_{\infty,A} + \|D^{(1)}V\|_{\infty,A} + \ldots + \|D^{(j)}V\|_{\infty,A}.
\end{gather*}

For $f: \Rn \to\R$, $x \in \Rn$ and $0 \le t < \infty$ we define the operator $T_t$ by
\begin{equation}\label{semigroup}
    T_t f(x) = \Ee\big[f(Y^x(t)); t < \zeta(x)\big].
\end{equation}

Before formulating the results for the process $Y(t)$ let us go
 back for a moment to the original problem
\eqref{sdemain}, that is,
$$
    \begin{cases}
        \displaystyle
        dY_t  =  dX_t + W'(Y_t) S_t\, dt, \\[\medskipamount]
        \displaystyle
        dS_t  =  W''(Y_t) \, dt.
    \end{cases}
$$
This SDE is of type \eqref{sdeY} because
 we can rewrite it as
\begin{equation}
\label{newsde}
    \begin{cases}
        \displaystyle
        dY_1(t)  =   dX_1(t) + V_1(Y(t))\, dt, \\[\medskipamount]
        \displaystyle
        dY_2(t)  =   dX_2(t) + V_2(Y(t))\, dt,
    \end{cases}
\end{equation}
where $X_1(t) = X_t$ is a symmetric $\alpha$-stable L\'evy process on $\R$, $\alpha \in (0,2)$, $X_2(t) \equiv 0$, $V_1(y_1,y_2) = W'(y_1)y_2$, $V_2(y_1,y_2) = W''(y_1)$. By Lemma \ref{existence} there exists a unique solution
to
 this SDE and the explosion time for this process is
infinite a.s.
 We want to show that $T_t f \in C_b^2$ whenever $f \in C^2_b$. Our
proof of Theorem \ref{Tt}
 requires that $V_i$ and its derivatives up to order 3 are bounded.
However,
 $V_1(y_1,y_2) = W'(y_1) y_2$ is not bounded on $\Rt$. We 
will
 circumvent this difficulty by
proving in Proposition \ref{coreTt}
 that $T_t f \in C_*^{2}(\Rt)$ whenever $f \in C_*^{2}(\Rt)$, where $C_*^2(\Rt)$ is given by Definition \ref{class}.

Let us briefly discuss the reasons that made us choose this particular set of
functions, $C_*^{2}(\Rt)$.
 This discussion gives also
an explanation for the specific assumptions in the main result of this section, Theorem \ref{Tt}.

Assume that $f \in C^2(\Rt)$ and $\supp{f} \subset K_0 = \R \times [-r,r]$, $r > 0$. Fix $t_0 < \infty$. If $|s| = |S_0| > r + t_0 \|W''\|_{\infty}$ then for $t \le t_0$,
$$
    \left|S_t^{(y,s)}\right| = \left|s + \int_{0}^{t} W''(Y_u^{(y,s)}) \, du\right| > r
$$
and, therefore,
$$
    T_t f(y,s) = \Ee f\big(Y_t^{(y,s)},S_t^{(y,s)}\big) = 0.
$$
It follows
 that if $t \le t_0$ then
\begin{equation}\label{suppK}
    \supp (T_t f) \subset K = \R \times \big[-r - t_0 \|W''\|_{\infty},\, r + t_0 \|W''\|_{\infty}\big].
\end{equation}
For technical reasons, we enlarge $K$ as follows,
$$
    K_3 = \R \times \big(-r - t_0 \|W''\|_{\infty} - 3, \, r + t_0 \|W''\|_{\infty} + 3\big).
$$
In view of \eqref{suppK}, we have to consider only starting points $(y,s) \in K$ in
order to prove that $T_t f \in C_*^2(\Rt)$. Note that for the starting point $(y,s) \in K_3$ and $t \le t_0$ we have
$$
    \left|S_t^{(y,s)}\right|
    = \left|s + \int_{0}^{t} W''(Y_u^{(y,s)}) \, du\right|
    \le  r + 2 t_0 \| W''\|_{\infty} +3.
$$
Thus for all starting points $(y,s) \in K_3$ and $t \le t_0$, 
\begin{equation}\label{M}
    \big(Y_t^{(y,s)},S_t^{(y,s)}\big)
    \in M
:=
 \R \times \big[-r - 2t_0 \|W''\|_{\infty} - 3, \, r + 2t_0 \|W''\|_{\infty} + 3\big].
\end{equation}
But the function
$V_1(y_1,y_2) = W'(y_1)y_2$ is bounded on $M$. Using our assumptions on $W$, namely, periodicity of $W$ and $W \in C^5$, we obtain also that the derivatives of $V_1(y_1,y_2) = W'(y_1) y_2$ up to order 3 are bounded on $M$.

\bigskip
Now we return to the general process $Y(t)$. Let us formulate the main result for this process.

\begin{theorem}\label{Tt}
Let $f: \Rn \to \R$ be a function in $C_b^2$.
 Fix $0 < t_0 < \infty$. Let $Y^x(t)$ be a solution of \eqref{sdeY}. Assume that the explosion time $\zeta(x,\omega) \equiv \infty$ for all $x \in \Rn$ and all $\omega \in \Omega$. Let $T_t f$ be defined by \eqref{semigroup}. Assume
that $K \subset \Rn$, for every $t \le t_0$ $\supp(T_t f) \subset K$
 and that there exists a convex set $M \subset \Rn$ such that
$Y^x(t,\omega) \in M$ for all $x \in K_3 := \bigcup_{x \in K} B^*(x,3)$, $t \le t_0$, and $\omega \in \Omega$. Assume that $\|V\|_{\infty,M} < \infty$ and $\|D^{(j)}V\|_{\infty,M} < \infty$ for $j = 1,2,3$.
 Then we have
$$
    T_t f \in C_b^2\quad\text{for all}\quad t\leq t_0.
$$
\end{theorem}

\begin{remark}\label{remTt}
When $\|V\|_{(3)} < \infty$
(i.e.\ when the assumptions of Theorem \ref{Tt} hold with $K = M = \Rn$)
 then the above theorem
 implies that we have for any $f \in C_b^2$
$$
    T_t f \in C_b^2\quad\text{for all}\quad t>0.
$$
\end{remark}

The first step in proving
Theorem \ref{Tt}
 will be the following proposition.
\begin{proposition}\label{path}
Fix $0 < t_0 < \infty$. Let $Y^x(t)$ be a solution of \eqref{sdeY}. Assume that the explosion time $\zeta(x,\omega) \equiv \infty$ for all $x \in \Rn$ and all $\omega \in \Omega$. Let $K \subset \Rn$. Assume that there exists a convex set $M \subset \Rn$ such that $Y^x(t,\omega) \in M$ for all $x \in K_3 := \bigcup_{x \in K} B^*(x,3)$, $t \le t_0$, and $\omega \in \Omega$. Assume that $\|V\|_{(3),M} < \infty$.
 Put
\begin{align}\label{5.4.1}
 \term
 := \frac{1}{2 \, \|D^{(1)}V\|_{\infty,M}} \wedge t_0,
    \qquad \Big(\frac 10 := \infty\Big).
\end{align}
For every $\omega \in \Omega$ we
have the following.
\begin{enumerate}
\item[\upshape (i)]
    For all $0 < t \le \term $, $x \in K_2 =\bigcup_{x \in K} B^*(x,2)$, $h \in \Rn$, $\|h\| < 1$,
\begin{align}\label{5.4.2}
        \|Y^{x + h}(t,\omega) - Y^x(t,\omega)\| \le 2 \|h\|.
\end{align}

\item[\upshape(ii)]
Recall that $e_i$ is the $i$-th unit vector in the usual orthonormal basis for $\R^n$.
    For all $0 < t \le \term $, $x \in K_2$, $i \in \{1,\ldots,n\}$,
    $$
        D_iY^x(t,\omega)
:=
 \lim_{u \to 0} \frac{Y^{x + u e_i}(t,\omega) - Y^x(t,\omega)}{u}
    $$
    exists, and
\begin{align}\label{5.4.3}
\|D_iY^x(t,\omega)\| \le 2.
\end{align}

    \noindent
    We will write
 $D_iY^x(t,\omega) = (D_iY_1^x(t,\omega),\ldots,D_iY_n^x(t,\omega))$.

\item[\upshape(iii)]
    For all $0 < t \le \term$, $x \in K_1 = \bigcup_{x \in K} B^*(x,1)$, $h \in \Rn$, $\|h\| < 1$, $i \in \{1,\ldots,n\}$,
\begin{align}\label{5.4.4}
        \|D_i Y^{x + h}(t,\omega) - D_i Y^x(t,\omega)\| \le 8\,\|D^{(2)}V\|_{\infty,M}\, \term  \,\|h\|.
\end{align}

\item[\upshape(iv)]
    For all $0 < t \le \term $, $x \in K_1$, $i,k \in \{1,\ldots,n\}$,
    $$
        D_{ik}Y^x(t,\omega)
:=
 \lim_{u \to 0} \frac{D_iY^{x + u e_k}(t,\omega) - D_iY^x(t,\omega)}{u}
    $$
    exists and
\begin{align}\label{5.4.5}
\|D_{ik}Y^x(t,\omega)\| \le 8\,
\|D^{(2)}V\|_{\infty,M}\, \term .
\end{align}

    \noindent
    We will write
 $D_{ik}Y^x(t,\omega) = (D_{ik}Y_1^x(t,\omega),\ldots,D_{ik}Y_n^x(t,\omega))$.

\item[\upshape(v)]
    For all $0 < t \le \term $, $x \in K$, $h \in \Rn$, $\|h\| < 1$, $i,k \in \{1,\ldots,n\}$,
    \begin{align*}
        \|D_{ik} Y^{x + h}&(t,\omega) - D_{ik} Y^x(t,\omega)\| \\
        &\le 96\,\|D^{(2)}V\|_{\infty,M}^2\, \term ^2\, \|h\| + 16\, \|D^{(3)}V\|_{\infty,M} \,\term  \,\|h\|.
    \end{align*}
\end{enumerate}
\end{proposition}

\begin{remark}
The existence of $D_i Y^x(t)$ and $D_{ik} Y^x(t)$ follows from \cite[Theorem V.40]{Pr}. What is new here are the explicit bounds for $D_i Y^x(t)$ and $D_{ik} Y^x(t)$ which are needed in the proof of Theorem \ref{Tt}, see Lemma \ref{Tt1}. The proof of Proposition \ref{path} is self-contained. We do not use \cite[Theorem V.40]{Pr}.
\end{remark}

\begin{proof}[Proof of Proposition \ref{path}]
The proof has a structure that might be amenable to presentation as a case of mathematical induction. After careful consideration we came to the conclusion that setting up an inductive argument would not shorten the proof.

Recall that we assume that \eqref{sdeY} holds for all $\omega \in \Omega$, not only a.s.
Throughout this proof we fix one path $\omega \in \Omega$.

\medskip\textbf{(i)}
Let $x \in K_2$, $h \in \Rn$, $\|h\|<1$ and $0 < t \le \term $. Recall that $X(0) = 0$. For any $1 \le j \le n$ we have
\begin{equation}\label{diffY}
    Y_j^{x + h}(t) - Y_j^x(t) = h_j + \int_0^t \left[V_j(Y^{x + h}(s)) - V_j(Y^x(s)) \right] ds.
\end{equation}
Let
$$
    c_1 := c_1(x,h)
 := \sup_{0 < t \le \term } \|Y^{x + h}(t) - Y^x(t)\|.
$$
Note that for $0 < t \le \term $ we have $Y^x(t) \in M$ and $Y^{x + h}(t) \in M$.
By \eqref{diffY} and $\|V\|_{\infty,M} < \infty$ we get that $c_1$ is finite. Moreover,
\begin{align*}
    \|Y_j^{x + h}(t) - Y_j^x(t)\|
    &\le \|h\| +  \int_0^t \|D^{(1)} V_j\|_{\infty,M}  \|Y^{x + h}(s) - Y^x(s)\| \, ds \\
    &\le \|h\| +  \,\term \, \|D^{(1)} V_j\|_{\infty,M} \,c_1.
\end{align*}
Hence,
$$
    c_1 \le \|h\| +  \term \, \|D^{(1)}V\|_{\infty,M} \, c_1,
$$
which, when combined with \eqref{5.4.1}, gives
$$
\sup_{0 < t \le \term } \|Y^{x + h}(t) - Y^x(t)\| =
    c_1 \le \frac{\|h\|}{1 -  \term  \,\|D^{(1)}V\|_{\infty,M}} \le 2 \|h\|.
$$

\medskip\textbf{(ii)}
Denote
$$
    R_j^{x,h}(t) = Y_j^{x + h}(t) - Y_j^x(t)
$$
and $R^{x,h}(t) = (R_1^{x,h}(t),\ldots,R_n^{x,h}(t))$. Using
the
 Taylor expansion we get from \eqref{diffY},
\begin{equation}\label{defR}
    R_j^{x,h}(t) = h_j + \int_0^t D^{(1)} V_j(Y^x(s))
\cdot
 R^{x,h}(s) \, ds + O(\|h\|^2).
\end{equation}
For $i \in \{1,\ldots,n\}$ and $h = u e_i$, let
$$
c_2 = c_2(x,i) = \max_{1 \le j \le n} \sup_{0 < t \le \term }
\left(\limsup_{u \to 0} \frac{R_j^{x,h}(t)}{u} - \liminf_{u \to 0} \frac{R_j^{x,h}(t)}{u}\right).
$$
Note that $ c_2 $ is finite because for $u \in (-1,1)$ we have $|R_j^{x,h}(t)| \le 2
u$, by \eqref{5.4.2}. Consider
 $0 < t \le \term $, $x \in K_2$, $i,j \in \{1,\ldots,n\}$. From \eqref{defR} we
obtain for
$u,u'\in (-1,1) \setminus \{0\}$, $h= ue_i$ and $h'=u'e_i$,
\begin{align*}
    \frac{R_j^{x,h}(t)}{u} - \frac{R_j^{x,h'}(t)}{u'}
    &= \int_0^t \sum_{k = 1}^n D_k V_j(Y^x(s)) \left(\frac{R_k^{x,h}(s)}{u} - \frac{R_k^{x,h'}(s)}{u'}\right) ds + O(u) + O(u').
\end{align*}
Letting $u,u'\to 0$ leads to
$$
    \limsup_{u\to 0}\frac{R_j^{x,h}(t)}{u} - \liminf_{u'\to 0}\frac{R_j^{x,h'}(t)}{u'}
    \leq \term \,\|D^{(1)} V\|_{\infty,M}\cdot c_2,
$$
and since $0 < t \le \term $ and $j \in \{1,\ldots,n\}$
are
 arbitrary, we get
$$
    c_2 \le \term \, \|D^{(1)}V\|_{\infty,M} \cdot c_2.
$$
So
 $c_2 = 0$ which means that $D_iY^x(t)$ exists.
Estimate \eqref{5.4.3} is now an easy consequence of \eqref{5.4.2}.

\medskip\textbf{(iii)}
From \eqref{defR} and the bounded convergence theorem, we obtain
\begin{equation}\label{defDiY}
    D_i Y_j^x(t) = \delta_{ij} + \int_0^t D^{(1)} V_j(Y^x(s)) \cdot D_i Y^x(s) \, ds.
\end{equation}
Let $x \in K_1$, $h \in \Rn$, $\|h\|<1$ and $i \in \{1,\ldots,n\}$. Set
$$
    c_3 := c_3(x,h,i) := \sup_{0 < t \le \term } \|D_i Y^{x + h}(t) - D_i Y^x(t)\|.
$$
Because of
\eqref{5.4.3},
 $c_3$ is finite. For any $0 < t \le \term $ we have
\begin{equation}\label{diffDiYj}\begin{aligned}
    D_i &Y_j^{x + h}(t) - D_i Y_j^x(t) \\
    &=\int_0^t \left[ D^{(1)} V_j(Y^{x+h}(s)) \cdot D_i Y^{x+h}(s) - D^{(1)} V_j(Y^x(s)) \cdot D_i Y^x(s) \right] ds \\
    &=\int_0^t \Big(\left[D^{(1)} V_j(Y^{x+h}(s)) - D^{(1)} V_j(Y^x(s))\right] \cdot D_i Y^{x+h}(s) \\
    &\qquad +  D^{(1)} V_j(Y^{x}(s)) \cdot \left[D_i Y^{x+h}(s) -  D_i Y^x(s)\right] \Big) ds,
\end{aligned}\end{equation}
so
\begin{align*}
    \left|D_i Y_j^{x + h}(t) - D_i Y_j^x(t)\right|
    &\le \int_0^t \left[ \sum_{k = 1}^n |D_k V_j(Y^{x+h}(s)) - D_k V_j(Y^x(s))|\ |D_i Y_k^{x+h}(s)|\right. \\
    &\qquad\qquad \left.+  \sum_{k = 1}^n |D_k V_j(Y^{x}(s))|
\ |D_i Y_k^{x+h}(s) -  D_i Y_k^x(s)| \right] ds.
\end{align*}
In view of \eqref{5.4.2} and \eqref{5.4.3}, we have for
$0 < s \le \term $,
\begin{align*}
    \sum_{k=1}^n |D_k V_j(Y^{x+h}(s)) - D_k V_j(Y^x(s))|
    &\le  \|D^{(2)}V\|_{\infty,M} \|Y^{x+h}(s) - Y^x(s)\| \\
    &\le 2 \, \|D^{(2)}V\|_{\infty,M} \|h\|,
\end{align*}
$$
    \|D_i Y^{x+h}(s)\| \le 2, \qquad   \sum_{k = 1}^n |D_k V_j(Y^{x}(s))|  \le \|D^{(1)}V\|_{\infty,M}.
$$
It follows that
$$
    |D_i Y_j^{x + h}(t) - D_i Y_j^x(t)|
    \le 4 \, \|D^{(2)}V\|_{\infty,M}\, \term \, \|h\| +  \term \, \|D^{(1)}V\|_{\infty,M}\cdot c_3,
$$
so,
$$
    c_3
    \le 4 \, \|D^{(2)}V\|_{\infty,M}\, \term \, \|h\| +  \term \, \|D^{(1)}V\|_{\infty,M} \cdot c_3.
$$
By definition, $\term \leq 1/(2 \|D^{(1)}V\|_{\infty,M})$, so
$$
    c_3
    \le 4 \, \|D^{(2)}V\|_{\infty,M}\, \term \, \|h\| + c_3/2.
$$
This gives
$$
\sup_{0 < t \le \term } \|D_i Y^{x + h}(t) - D_i Y^x(t)\| =
    c_3 \le 8 \, \|D^{(2)}V\|_{\infty,M}\, \term \, \|h\|.
$$

\medskip\textbf{(iv)}
Set
$$
    Q_{i,j}^{x,h}(t) := D_i Y_j^{x + h}(t) - D_i Y_j^x(t)
$$
and $Q_i^{x,h}(t) = (Q_{i,1}^{x,h}(t),\ldots,Q_{i,n}^{x,h}(t))$. Using the
 Taylor expansion we get from \eqref{diffDiYj},
\begin{equation}\label{Qij1}\begin{aligned}
    Q_{i,j}^{x,h}(t)
    &= \int_0^t  \sum_{l = 1}^n D_i Y_l^{x + h}(s) \sum_{m = 1}^n D_{lm} V_j(Y^x(s)) R_m^{x,h}(s) \, ds + O(\|h\|^2) \\
    &\qquad+\int_0^t D^{(1)} V_j(Y^x(s)) \cdot Q_i^{x,h}(s) \, ds \\
    &=\int_0^t  \sum_{l = 1}^n D_i Y_l^{x + h}(s)  D^{(1)} D_{l} V_j(Y^x(s))\cdot R^{x,h}(s) \, ds + O(\|h\|^2) \\
    &\qquad+\int_0^t D^{(1)} V_j(Y^x(s)) \cdot Q_i^{x,h}(s) \, ds.
\end{aligned}\end{equation}
Consider
 $k \in \{1,\ldots,n\}$ and let $h = u e_k$. Define
$$
    c_4
    := c_4(x,i,k)
    := \max_{1 \le j \le n} \sup_{0 < t \le \term } \left(\limsup_{u \to 0} \frac{Q_{i,j}^{x,h}(t)}{u} - \liminf_{u \to 0} \frac{Q_{i,j}^{x,h}(t)}{u}\right).
$$
Note that $c_4$ is finite because we have $|Q_{i,j}^{x,h}(t)| \le 8 \, \|D^{(2)}V\|_{\infty,M} \, \term \, u$ for $u \in  (-1,1)$,
by \eqref{5.4.4}.
 For $u,u'\in (-1,1) \setminus \{0\}$,
$h=ue_k$ and $h'=u'e_k$, \eqref{Qij1} implies that,
\begin{align*}
    \frac{Q_{i,j}^{x,h}(t)}{u} - \frac{Q_{i,j}^{x,h'}(t)}{u'}
    &= \int_0^t  \sum_{l = 1}^n D_i Y_l^{x + h}(s)  D^{(1)} D_{l} V_j(Y^x(s))
\,\cdot\,
 \frac{R^{x,h}(s)}{u} \, ds + O(u) \\
    &\quad - \int_0^t  \sum_{l = 1}^n D_i Y_l^{x + h'}(s)  D^{(1)} D_{l}
V_j(Y^x(s)) \,\cdot\,
 \frac{R^{x,h'}(s)}{u'} \, ds + O(u')\\
    &\quad + \int_0^t D^{(1)} V_j(Y^x(s)) \left(\frac{Q_i^{x,h}(s)}{u} - \frac{Q_i^{x,h'}(s)}{u'} \right) ds.
\end{align*}
The first two integrals cancel in the limit as $u,u'\to 0$. To see that we can pass to the limit, we use the bounded convergence theorem. This theorem is applicable because \eqref{5.4.2} provides a bound for $\frac 1u\,R^{x,h}(s)$, \eqref{5.4.3} provides a bound for $D_i Y_l^{x + h}(s)$ and we also have $\|D^{(2)} V\|_{\infty,M} <\infty$, by assumption. Letting $u,u'\to 0$ we get
$$
    \limsup_{u \to 0} \frac{Q_{i,j}^{x,h}(t)}{u} - \liminf_{u' \to 0} \frac{Q_{i,j}^{x,h'}(t)}{u'}
    \le \term \,  \|D^{(1)}V\|_{\infty,M}\cdot c_4.
$$
Since $0 < t \le \term $ and $j \in \{1,\ldots,n\}$
are arbitrary we see that
$$
    c_4
    \le \term  \, \|D^{(1)}V\|_{\infty,M}\cdot c_4,
$$
so $c_4 = 0$; this proves that $D_{i k} Y^x(t)$ exists. The estimate \eqref{5.4.5} follows now from \eqref{5.4.4}.

\medskip\textbf{(v)}
By \eqref{Qij1} we get  for $h = u e_k$
\begin{align*}
    D_{ik} Y_j^x(t)
    = \lim_{u \to 0} \frac{Q_{i,j}^{x,h}(t)}{\|h\|}
    &= \int_0^t  \sum_{l = 1}^n D_i Y_l^{x}(s)\,  D^{(1)} D_{l} V_j(Y^x(s)) \cdot D_k Y^x(s) \, ds \\
    &\qquad+ \int_0^t D^{(1)} V_j(Y^x(s)) D_{ik} Y^x(s) \, ds.
\end{align*}
Let $x \in K$, $h \in \Rn$, $\|h\|<1$ and $i,k \in \{1,\ldots,n\}$. Put
$$
    c_5 := c_5(x,h,i,k) := \sup_{0 < t \le \term } \|D_{ik} Y^{x + h}(t) - D_{ik} Y^x(t)\|.
$$
Because of
\eqref{5.4.5},
 $c_5$ is finite. For any $0 < t \le \term $ and $j \in \{1,\ldots,n\}$ we have
\begin{align*}
    D_{ik} &Y_j^{x + h}(t) - D_{ik} Y_j^x(t) \\
    &= \int_0^t  \sum_{l = 1}^n \sum_{m = 1}^n \Big[ D_i Y_l^{x + h}(s) D_{lm} V_j(Y^{x + h}(s)) D_k Y_m^{x + h}(s) \\
    &\qquad- D_i Y_l^{x}(s) D_{lm} V_j(Y^{x}(s)) D_k Y_m^{x}(s) \Big] ds \\
    &\qquad+ \int_0^t \sum_{l = 1}^n \Big[D_l V_j(Y^{x + h}(s)) D_{ik} Y_l^{x + h}(s) - D_l V_j(Y^{x}(s)) D_{ik} Y_l^{x}(s) \Big] ds \\
    &= \text{I} + \text{II}.
\end{align*}
We obtain from \eqref{5.4.3}, \eqref{5.4.4} and \eqref{5.4.5},
\begin{align*}
    |\text{I}|
    &\le  \int_0^t  \sum_{l = 1}^n \sum_{m = 1}^n \Big[\left|D_{lm} V_j(Y^{x + h}(s)) D_i Y_l^{x + h}(s) \left[D_k Y_m^{x + h}(s) - D_k Y_m^{x}(s)\right]\right| \\
    &\qquad+ \left|D_{lm} V_j(Y^{x + h}(s)) D_k Y_m^{x}(s) \left[D_i Y_l^{x + h}(s) - D_i Y_l^{x}(s)\right]\right| \\
    &\qquad+ \left|D_i Y_l^{x}(s) D_k Y_m^{x}(s) \left[D_{lm} V_j(Y^{x + h}(s)) - D_{lm} V_j(Y^{x}(s))\right]\right| \Big] ds \\
    &\le \term  \,  \Big[\|D^{(2)}V\|_{\infty,M}^2\, 32 \, \term \, \|h\| + 8 \, \|D^{(3)}V\|_{\infty,M} \, \|h\| \Big],
\end{align*}
as well as
\begin{align*}
    |\text{II}|
    &\le \int_0^t \sum_{l = 1}^n \Big[\left|D_l V_j(Y^{x + h}(s)) \left[D_{ik} Y_l^{x + h}(s) - D_{ik} Y_l^{x}(s)\right] \right| \\
    &\qquad\qquad+ \left|D_{ik} Y_l^{x}(s) \left[D_l V_j(Y^{x + h}(s)) - D_l V_j(Y^{x}(s))\right]\right| \Big] ds \\
    &\le \term \, \Big[\|D^{(1)}V\|_{\infty,M}\cdot c_5 + 16 \, \|D^{(2)}V\|_{\infty,M}^2 \term \,  \|h\|\Big].
\end{align*}
Combining these two estimates we find for all $0 < t \le \term $ and $1 \le j \le n$,
\begin{align*}
    |D_{ik} & Y_j^{x + h}(t) - D_{ik} Y_j^x(t)| \\
    &\le 48  \,\|D^{(2)}V\|_{\infty,M}^2\, \term ^2 \|h\| + 8  \,\|D^{(3)}V\|_{\infty,M}\term  \,\|h\| + \term \, \|D^{(1)}V\|_{\infty,M} \cdot c_5.
\end{align*}
Hence,
$$
    c_5 \le 48 \, \|D^{(2)}V\|_{\infty,M}^2\, \term ^2 \,\|h\| + 8 \, \|D^{(3)}V\|_{\infty,M}\,\term \, \|h\| + \term \,  \|D^{(1)}V\|_{\infty,M}\cdot c_5,
$$
so, recalling \eqref{5.4.1},
$$
    c_5 \le 96  \, \|D^{(2)}V\|_{\infty,M}^2\, \term ^2\, \|h\| + 16 \, \|D^{(3)}V\|_{\infty,M}\,\term \, \|h\|,
$$
which finishes the proof.
\end{proof}
The next step in proving Theorem \ref{Tt} is the following lemma.

\begin{lemma}\label{Tt1}
Let $g: \Rn \to \R$ be
a function in $C_b^2$.
 Fix $0 < t_1 < \infty$ and let $Y^x(t)$ be the solution of \eqref{sdeY}. Assume that the explosion time $\zeta(x,\omega) \equiv \infty$ for all $x \in \Rn$ and all $\omega \in \Omega$. Let $T_t g$ be defined by \eqref{semigroup}. Assume
that $K \subset \Rn$,
 for every $t \le t_1$ $\supp{T_t g} \subset K$ and
 there exists a convex set $M \subset \Rn$ such that $Y^x(t,\omega) \in M$ for all $x \in K_3 := \bigcup_{x \in K} B^*(x,3)$, $t \le t_1$ and $ \omega \in \Omega$.
Assume that $\|V\|_{(3),M} < \infty$ and let
$$
    \tilde \term  = \frac{1}{2 \, \|D^{(1)}V\|_{\infty,M}} \wedge t_1
    \qquad \Big(\frac 10 := \infty\Big).
$$
Then we have
\begin{enumerate}
    \item[\upshape (i)]
    For all $0 < t \le \tilde \term $, $x \in K$
and $i \in \{1,\ldots,n\}$, the derivative
 $D_i T_t g(x)$ exists and
    \begin{equation}\label{Tt1a}
        D_i T_t g(x) = \Ee\left(D^{(1)} g(Y^x(t)) D_iY^x(t)\right).
    \end{equation}

\item[\upshape (ii)]
    For all $0 < t \le \tilde \term $, $x \in K$
and $i,k \in \{1,\ldots,n\}$, the derivative
 $D_{ik} T_t g(x)$ exists and
    \begin{align}\label{Tt1b}
        &D_{ik} T_t g(x) \\
        &=
 \Ee\left(D^{(1)} g(Y^x(t))
\,\cdot \, D_{ik}Y^x(t)
 + \sum_{j = 1}^n D_iY_j^x(t) D^{(1)}(D_jg )(Y^x(t))\,\cdot\, D_kY^x(t)
 \right). \nonumber
    \end{align}

\item[\upshape (iii)]
    For all $0 < t \le \tilde \term $
and $i,k \in \{1,\ldots,n\}$, the derivative
 $D_{ik} T_t g(x)$ is continuous for $x \in K$.
\end{enumerate}
\end{lemma}

\begin{proof}
\textbf{(i)}
Let $0 < t \le \tilde \term $, $x \in K$, fix $i \in \{1,\ldots,n\}$ and let $h = u e_i$. By Taylor's theorem
and  \eqref{5.4.2}, we get,
\begin{align*}
    D_i T_t g(x)
    &= \lim_{u \to 0} \frac{T_t g(x + h) - T_t g(x)}{u} \\
    &= \lim_{u \to 0} \Ee\left(\frac{g(Y^{x+h}(t)) - g(Y^x(t))}{u}\right) \\
    &= \lim_{u \to 0} \Ee\left(\frac{D^{(1)} g(Y^x(t)) \cdot (Y^{x + h}(t) -Y^x(t))}{u}\right) \\
    &\quad+ \lim_{u \to 0} \Ee\left(\frac{\sum_{1\le l,m \le n}D_{lm}g(\xi)
 (Y_l^{x+h}(t) - Y_l^x(t))(Y_m^{x + h}(t) - Y_m^x(t))}{2 u}\right) \\
    &= \Ee\left(D^{(1)} g (Y^x(t))\cdot D_i Y^x(t)\right) + \lim_{u \to 0} \Ee\left(O\left(\frac{\|Y^{x + h}(t) - Y^x(t)\|^2}{u}\right)\right) \\
    &= \Ee\left(D^{(1)} g (Y^x(t)) \cdot D_i Y^x(t)\right),
\end{align*}
where $\xi = \xi_{x,h,t,l,m}$ is an intermediate point between $Y^{x}(t)$ and $Y^{x+h}(t)$. This 
yields \eqref{Tt1a}.

\medskip\textbf{(ii)}
Fix $i,k \in \{1,\ldots,n\}$ and let $h = u e_k$. We have, using (i),
\begin{align*}
    D_{ik} T_t g(x)
    &=\lim_{u \to 0} \frac{D_i T_t g(x + h) - D_i T_t g(x)}{u} \\
    &=\lim_{u \to 0} \Ee \left(\frac{D^{(1)}g (Y^{x + h}(t))\cdot D_i Y^{x + h}(t) - D^{(1)} g (Y^x(t))\cdot D_i Y^x(t)}{u}\right) \\
    &=\lim_{u \to 0} \Ee \left(\frac{D^{(1)} g (Y^{x + h}(t))\cdot (D_i Y^{x + h}(t) -  D_i Y^x(t))}{u}\right) \\
    &\qquad+ \lim_{u \to 0} \Ee \left(\frac{D_i Y^x(t)\cdot (D^{(1)} g (Y^{x + h}(t)) - D^{(1)} g (Y^x(t)))}{u}\right) \\
    &= \text{I} + \text{II}.
\end{align*}
By \eqref{5.4.4} and bounded convergence theorem,
$$
    \text{I} = \Ee\left(D^{(1)} g (Y^x(t)) \cdot D_{ik} Y^x(t)\right).
$$
We apply the Taylor theorem, \eqref{5.4.2} and the bounded convergence theorem to see that
\begin{align*}
    \text{II}
    &= \lim_{u \to 0} \Ee \left( \frac{\sum_{j = 1}^n D_i Y_j^x(t) (D_j g (Y^{x + h}(t)) - D_j g (Y^x(t)))}{u}\right) \\
    &= \lim_{u \to 0} \Ee \left( \frac{\sum_{j = 1}^n D_i Y_j^x(t) D^{(1)}(D_j g)(Y^x(t)) \cdot (Y^{x + h}(t) - Y^x(t))}{u}\right) \\
    &\qquad+ \lim_{u \to 0} \Ee \left(O\left(\frac{\|Y^{x + h}(t) - Y^x(t)\|^2}{u}\right) \right) \\
    &= \Ee\left(\sum_{j = 1}^n D_i Y_j^x(t) D^{(1)}(D_j g)(Y^x(t)) \cdot D_k Y^x(t)\right).
\end{align*}
This proves \eqref{Tt1b}.

\medskip\textbf{(iii)}
By Proposition \ref{path}, all derivatives on the right hand side of \eqref{Tt1b} are continuous. Thus the function $D_{ik}T_t g(x)$ is continuous for $x \in K$, $i,k \in \{1,\ldots,n\}$ and $0 < t \le \tilde \term $. This proves (iii).
\end{proof}

\begin{proof}[Proof of Theorem \ref{Tt}]
We set
$$
    \term  := \frac{1}{2\, \|D^{(1)}V\|_{\infty,M}} \wedge t_0.
$$
We will use induction. The induction step is the following.
Assume that $T_s f \in C_b^2$ for some $s \in [0,t_0]$. We will
 show that for all $r \le \term $ such that $s + r \le t_0$ we have $T_{s + r}f \in C^2$ and $\|T_{s + r}f\|_{(2)} < \infty$. To show this we use Lemma \ref{Tt1}. Put $g = T_s f$ and $t_1 = t_0 - s$. Note that $r \le \term  \wedge t_1 = \tilde \term $ and $g = T_s f$ satisfies the assumptions of Lemma \ref{Tt1}. Hence we obtain that $T_{r + s} f = T_r g \in C^2$. A combination of the estimates \eqref{Tt1a}, \eqref{Tt1b}, the fact that $\supp{T_r g} \subset K$ and the estimates from Proposition \ref{path} yield $\|T_rg\|_{(2)} < \infty$.

An assumption of Theorem \ref{Tt} states that $f\in C^2_b$. Hence, $T_0 f =f \in C^2_b$. The induction step shows that $T_s f  \in C^2_b$ for all $s \leq \term \land t_0$. Subsequent induction steps extend this claim to $T_s f \in C^2_b$ for all $s \leq j\term \land t_0$, $j=2,3, \dots$ Therefore,  $T_{s} f \in C^2_b$ for all $s \le t_0$.
\end{proof}

\begin{proposition}
\label{coreTt}
Let $\{T_t\}_{t \ge 0}$ be the semigroup given by \eqref{semigroupZS} of the the process $(Y_t,S_t)$ defined by \eqref{sdemain}. Let $C_*^2(\Rt)$ be the class of functions given by Definition \ref{class}. We have
$$
T_t: \, C_*^2(\Rt) \to C_*^2(\Rt).
$$
\end{proposition}
\begin{proof}
We will repeat some of the arguments given before the statement of Theorem \ref{Tt}. Note that the SDE \eqref{sdemain} is of the form \eqref{sdeY}. By Lemma \ref{existence} there exists a unique solution of \eqref{sdemain} with explosion time $\zeta((y,s),\omega) \equiv \infty$ for all $(y,s) \in \Rt$ and $\omega \in \Omega$. Suppose that
 $f \in C_*^2(\Rt)$. Then $\supp{f} \subset \R \times [-r,r]$, for some $r > 0$. Fix $t_0 > 0$. By \eqref{suppK}, for any $t \le t_0$, we have
\begin{align}\label{5.4.10}
    \supp{T_t f} \subset K := \R \times \big[-r - t_0 \|W''\|_{\infty}, \, r + t_0 \|W''\|_{\infty}\big].
\end{align}
We have
$$
    K_3 = \bigcup_{(y,s) \in K}B^*((y,s),3) = \R \times \big(-r - t_0 \|W''\|_{\infty} - 3, \, r + t_0 \|W''\|_{\infty} + 3\big).
$$
Let
$$
    M = \R \times \big[-r - 2t_0 \|W''\|_{\infty} - 3, \,r + 2t_0 \|W''\|_{\infty} + 3\big].
$$
By \eqref{M} we have $(Y_t^{(y,s)},S_t^{(y,s)}) \in M$ for all $(y,s) \in K_3$. Rewriting \eqref{sdemain} as \eqref{newsde} we have $V_1(y_1,y_2) = W'(y_1)y_2$, $V_2(y_1,y_2) = W''(y_1)$. Since $W \in C^5$ and since it is periodic, we get $\|V\|_{(3),M} < \infty$. Therefore, the solution of \eqref{newsde} satisfies the assumptions of Theorem \ref{Tt}. It follows that for any $t \le t_0$ we have
$$
    T_t f \in C^2 \quad\text{and}\quad \|T_t f\|_{(2)} < \infty.
$$
This and \eqref{5.4.10} yield
 $T_t f \in C_*^2(\Rt)$.
\end{proof}

\begin{proof}[Proof of Proposition \ref{TtD}]
Suppose that
 $f \in C_c^2(\uni \times \R)$. Then $\tilde{f} \in C_*^2(\Rt)$ where $\tilde{f}$ is given by \eqref{ftilde2}. By Proposition \ref{coreTt}, $T_t \tilde{f} \in C_*^2(\Rt)$. Using this and \eqref{doubletilde} we get $T_t^\uni f \in C_c^2(\uni \times \R)$.
\end{proof}

\end{document}